\documentclass{article}
\usepackage{amsmath,graphicx}
\usepackage{subcaption}
\usepackage{amsthm}
\usepackage{amssymb}
\usepackage[all]{xy}
\usepackage{xcolor}
\theoremstyle{plain}
\newtheorem{theorem}{Theorem}
\newtheorem{proposition}[theorem]{Proposition}
\newtheorem{corollary}[theorem]{Corollary}
\newtheorem{lemma}[theorem]{Lemma}
\theoremstyle{definition}
\newtheorem{definition}[theorem]{Definition}

\newtheorem{remark}[theorem]{Remark}
\newtheorem{example}[theorem]{Example}

\def\shf{\mathcal}
\def\col{\mathcal}

\def\rank{\textrm{rank }}
\def\image{\textrm{image }}
\def\st{\textrm{star }}
\def\cl{\textrm{cl }}

\def\fr{\textrm{fr }}
\def\lk{\textrm{lk }}
\def\interior{\textrm{int }}

\title{Local homology of abstract simplicial complexes}
\author{Michael Robinson \and Chris Capraro \and Cliff Joslyn \and Emilie Purvine \and Brenda Praggastis \and Stephen Ranshous   \and Arun Sathanur}

\begin{document}
\maketitle
\begin{abstract}
This survey describes some useful properties of the local homology of abstract simplicial complexes.  Although the existing literature on local homology is somewhat dispersed, it is largely dedicated to the study of manifolds, submanifolds, or samplings thereof.  While this is a vital perspective, the focus of this survey is squarely on the local homology of abstract simplicial complexes.  Our motivation comes from the needs of the analysis of hypergraphs and graphs.  In addition to presenting many classical facts in a unified way, this survey presents a few new results about how local homology generalizes useful tools from graph theory.  The survey ends with a statistical comparison of graph invariants with local homology.
\end{abstract}

\tableofcontents

\section{Introduction}

This survey describes some useful properties of the local homology of abstract simplicial complexes.  Although the existing literature on local homology is somewhat dispersed, it is largely dedicated to the study of manifolds, submanifolds, or samplings thereof.  While this is a vital perspective -- and is one that we do not ignore here -- our focus is squarely on the local homology of abstract simplicial complexes. Our motivation comes from the needs of the analysis of hypergraphs and graphs.  We note that although there are few purely topological invariants of graphs, namely connected components, loops, and vertex degree, the topology of abstract simplicial complexes is substantially richer.  Abstract simplicial complexes are becoming more frequently used in applications, and purely topological invariants of them are both expressive and insightful.  

Judging by the literature, most of the attention on abstract simplicial complexes falls in two areas: (1) their construction from data, and (2) their global analysis using homological tools.  Although homology itself is sensitive to outliers, persistent homology \cite{Edelsbrunner_2002} is provably and practically robust.  It is for this reason that persistent homology is ascendant among recent topological tools.  

But, persistent homology is by nature global, and sometimes this is not desirable.  There is a lesser-known variant of homology, called \emph{local homology} that is also an expressive tool for studying topological spaces.  It captures a surprising variety of useful topological properties:
\begin{enumerate}
 \item For graphs -- isomorphic to 1-dimensional simplicial complexes -- it detects \emph{graph degree} (Proposition \ref{prop:graph_degree} in Section \ref{sec:graphs}),
 \item It is bounded by \emph{local clustering coefficient} in planar graphs (Theorem \ref{thm:LH_CC_bounds} in Section \ref{sec:graph_results}),
 \item It detects the \emph{dimension} of triangulated smooth manifolds (Proposition \ref{prop:homology_manifold} in Section \ref{sec:stratification_detection}),
 \item It detects \emph{boundaries} of triangulated manifolds, if they have them (Proposition \ref{prop:manifold_boundary} in Section \ref{sec:stratification_detection}),
 \item More generally, it detects cells representing non-manifold \emph{strata} (Section \ref{sec:stratification_detection}).
\end{enumerate}
A number of researchers \cite{Brown_2017,Fasy_2016,Dey_2014,Dey_2013,Bendich_2007} have recently explored these properties for point clouds derived from embedded submanifolds, primarily motivated by the concerns of manifold learning.  However, all of the above properties are \emph{intrinsic}, and do not rely on a given embedding.  This survey aims to close this gap, by providing an intrinsic, combinatorial look at both the properties (Sections \ref{sec:theory} -- \ref{sec:general}) and the pragmatics (Sections \ref{sec:computation} and \ref{sec:statistics}) of computing local homology of abstract simplicial complexes.  Our aim is twofold: first to showcase these intrinsic properties in their ``natural habitat'' and second to advocate for their use in applications.  In service to the latter, in Section \ref{sec:computation} we discuss computational aspects of one local homology library {\tt pysheaf} \cite{pysheaf} that we are actively developing, and demonstrate results on well-studied benchmark graph datasets in Section \ref{sec:statistics}.

Since local homology has been studied for over eighty years \cite{Alexandrov_1935}, many things about it are known, but the literature is disappointingly diffuse.  This article draws the related threads of knowledge together under the banner of abstract simplicial complexes, as opposed to general topological spaces or (stratified) manifolds.  In the context of computation and applications, there is a strong connection to sheaf theory.  Local homology is derived from the global sections of the \emph{homology sheaf}, which can be constructed rather concretely on abstract simplicial complexes (Proposition \ref{prop:homology_sheaf} in Section \ref{sec:local_homology_def}).  Regrettably, this sheaf-theoretic viewpoint is not as powerful as one might hope, since simplicial maps do not induce\footnote{Local \emph{cohomology} appears to have the desired functoriality, and can be used to generalize the concept of \emph{degree} of a continuous map \cite{olum1953mappings}.  However, local cohomology is not a sheaf -- it is something of a ``partial'' precosheaf according to Proposition \ref{prop:functorish} in Section \ref{sec:neighborhood}.} sheaf morphisms between homology sheaves (Example \ref{eg:functorfail} in Section \ref{sec:neighborhood}).  This perhaps explains why homology sheaves are not as prevalent in applications as one might suspect.

In addition to presenting many classical facts about local homology in Section \ref{sec:theory}, we also present a few new ideas. 
\begin{enumerate}
\item In Section \ref{sec:graphs}, we present some new results on how the clustering coefficient of a planar graph is related to its local homology,
\item In Section \ref{sec:general}, we show that the first local Betti number generalizes the degree of a vertex in a graph, and use this to interpret the other local Betti numbers as generalized degrees of other simplices in an abstract simplicial complex,
\item In Section \ref{sec:computation}, we discuss an efficient computational algorithm for local homology, tailored specifically to abstract simplicial complexes, and
\item In Section \ref{sec:statistics}, we discuss certain correlations between local Betti numbers on graphs and several graph invariants used in network science.
\end{enumerate}

\section{Historical discussion}
\label{sec:history}

The concept of local homology springs from the work of \v{C}ech \cite{Cech_1934} and Alexandov \cite{Alexandrov_1935} on Betti numbers localized to a point in the early 1930s.  Local homology must have been on the minds of both for some time, since \v{C}ech credits Alexandrov in his introduction, and Alexandrov had published some of the ideas earlier \cite{Alexandrov_1929,Alexandrov_1932,Alexandrov_1933}.   Alexandrov's restatement of the definition of local Betti numbers at a point using the then-new idea of relative homology provided the right way to greater generality.  

Based on Alexandrov's constructions, Steenrod \cite{steenrod1943homology} wrote a survey of local methods in topology a few years later, which includes the combinatorial construction that we use in Section \ref{sec:theory}.  Since his focus was squarely on topological spaces generally and manifolds in particular, Steenrod does not spend much time on his combinatorial definition, and includes no discussion of its implications.  He does recognize that local homology forms a sheaf, a fact he had proved the year before \cite{Steenrod_1942}.   This was one of the earliest concrete constructions of a sheaf; one imagines that Steenrod's and Leray's work on sheaves were happening in parallel, during Leray's captivity \cite{miller2000leray}.  Steenrod called his sheaf a \emph{system of local coefficients}, but following Borel \cite{Borel_1957}, the sheaf of local homology is now usually called the \emph{homology sheaf}.  

Borel (later working with Moore \cite{Borel_1960}) used the homology sheaf to prove Poincar\'{e} duality theorems for a number of classes of generalized manifolds.  This is apparently not a historical accident, as the study of local homology was intimately knit into the discovery of the correct way to generalize manifolds.  It was known quite early \cite{Alexandrov_1932} that local homology can be used to compute the dimension of a space, and that this definition agreed with the definition of a manifold.  In his book, Wilder \cite{Wilder_1949} used Alexandov's definition of local Betti numbers at a point (a concept subsequently generalized by White \cite{White_1950,White_1952} to closed sets, essentially mirroring Steenrod's construction using a direct limit in the homology sheaf) to constrain the neighborhoods of points.  Although there is considerable subtlety in Wilder's generalized manifold definition, Bredon showed that Wilder's generalized manifolds are locally orientable \cite{Bredon_1969} using the homology sheaf.

Milnor and Stasheff used local homology to examine the orientation of vector bundles in their classic book \cite[Appendix A]{MilnorStasheff_1974}.  They also have a result relating local homology to the induced orientation of boundaries, which is a reflection of its power in \emph{non-manifold} spaces.  The relationship between the orientation of a space and its boundary has continued to require the study of local homology for more general spaces.  For instance, Mitchell \cite{Mitchell_1990} used local homology to characterize the boundaries of homology manifolds.

Local homology is also discussed at various points in Munkres' classic textbook on algebraic topology \cite{Munkres_1984}.  Although Munkres uses abstract simplicial complexes in his book, his focus is mostly on using them as a convenient representation for working with topological spaces.  Therefore most of his statements are in terms of geometric realizations of abstract simplicial complexes.  However, he provides concise proofs of a number of facts that will be useful to the discussion in this survey, including that local homology is locally constant within the interior of a simplex \cite[Lem 35.2]{Munkres_1984} and that it provides a way to identify stratifications (Proposition \ref{prop:homology_manifold} in Section \ref{sec:stratification_detection}).  

That local homology has something to do with stratifications in simplicial complexes suggested that it has deeper theoretical analogues.  Goresky and MacPherson \cite{Goresky_1988} showed that stratified spaces can be effectively studied using intersection homology.  It is straightforward to show that local homology is a special case of intersection homology, and that is especially clear for simplicial spaces \cite{Bendich_2007}.  More generally, Rourke and Sanderson used local homology to examine stratified spaces in detail from a theoretical level \cite{rourke1999homology}.  Intersection homology even has a robust, persistent version for abstract simplicial complexes as shown by Bendich and Harer \cite{bendich2011persistent}.

Growing primarily out of the initial work in Bendich's thesis \cite{Bendich_2008}, the modern computational study of local homology has focused on the local persistent homology of point clouds.  There have been a number of fruitful directions, namely
\begin{enumerate}
\item Those following the fundamental results proven by Bendich and his collaborators \cite{bendich2012local,Bendich_2008,Bendich_2007}
\item Witnessed filtrations of Vietoris-Rips complexes to aid in more efficient computation \cite{skraba2014approximating}
\item Studying filtrations of general covers \cite{Fraser_2012}
\item Connecting local homology to dimension reduction and traditional manifold learning approaches \cite{Brown_2017,Dey_2014,Dey_2013}, especially because local homology is not the only way to learn stratified manifolds (see for instance \cite{haro2006stratification}, which uses expectation maximization),
\item Connecting local homology to exploratory data analysis \cite{Fasy_2016, ahmed2014local}.
\end{enumerate}
As noted in the introduction, this survey focuses on intrinsic local homology rather than embedded point clouds, if for no other reason that this seems to have unexplored merit in exploratory data analysis. 
Abstract simplicial complexes provide minimal topological environments on which to construct local homologies. Moreover, as demonstrated in \cite{clader2009inverse,mccord1966} and discussed in \cite{barmak2011algebraic,May_2003}, finite abstract simplicial complexes are weakly homotopic to their geometric realizations, indicating a study of the former will reveal information about the homotopy invariants of the latter \cite{Stong_1966}.

We end this brief historical discussion by noting that there is a concept dual to local homology -- that of \emph{local cohomology}.  Since the 1950s, local cohomology of topological spaces \cite{Raymond_1961} has been known to generalize the notion of a \emph{degree} of a smooth map \cite{olum1953mappings}, and that fixed points of a smooth mapping are classified induced maps on local cohomology.  However, the local cohomology of \emph{spaces} appears to have a much smaller following than the local cohomology of \emph{algebraic objects}, due to Grothendieck's vast generalization \cite{Grothendieck_1967}.  Because these two concepts of local cohomology are manifestly similar, it may be argued that local cohomology is more natural than local homology (we recommend the survey \cite{Brodman_1998} on local cohomology in algebraic geometry).  We note that computational aspects of both local homology and cohomology are presently fairly immature, but most applications are currently easier to interpret in the context of local homology.  

\section{Theoretical groundwork}
\label{sec:theory}
This article studies the local homology of abstract simplicial complexes.  Some computational efficiency can be gained by using other kinds of cell complexes, though their use complicates the exposition.

\noindent
\begin{definition} Let $V$ be a countable set. An \emph{abstract simplicial complex} $X$ with \emph{vertices} in $V$ is a collection of finite subsets of $V$ such that if $\sigma \in X$ and $\tau \subseteq \sigma$ then $\tau \in X$.  An element $\sigma$ of $X$ is called a \emph{simplex} or \emph{face of $X$}. A simplex $\sigma$ has \emph{dimension} equal to $|\sigma| - 1$. The \emph{dimension of $X$} is the maximal dimension of its simplices. We will represent each $\sigma \in X$  as a bracketed list\footnote{The order of the list is somewhat arbitrary but provides a helpful notation and is needed when computing homology.} of vertices: $\sigma = [v_0,\dotsc,v_k]$. A subset $Y \subseteq X$ is called a \emph{subcomplex} if it is an abstract simplicial complex in its own right.
\end{definition}

\begin{definition}
An abstract simplicial complex $X$ comes equipped with a natural topology, called the \emph{Alexandrov topology} \cite{Alexandrov_1937}, whose open sets are composed of arbitrary unions of sets of the form
\begin{equation*}
\st \sigma = \{\tau\in X : \sigma \subseteq \tau\}
\end{equation*}
where $\sigma$ is a face of $X$.  We shall assume that all abstract simplicial complexes are \emph{locally finite}, which means that all stars over simplices are finite sets.  
The Alexandrov topology induces a partial ordering on $X$ given by $\sigma \leq \tau$ if and only if $\st \sigma \subseteq \st \tau$.  It follows that $\sigma \leq \tau$ if and only if $\tau \subset \sigma$.\footnote{The Alexandrov topology actually induces two possible partial orders on the simplices of $X$ depending on the direction of inclusion in the definition of the star. Additionally, the Alexandrov topology can be built from a pre-order (not necessarily a partial order).  Within the context of abstract simplicial complexes, partial orders suffice.}
\end{definition}

\begin{lemma}
\label{lem:aspace}
The Alexandrov topology $\col{T}$ for an abstract simplicial complex $X$ makes $(X,\col{T})$ into an \emph{Alexandrov space}, namely one which is closed under arbitrary intersections.
\end{lemma}

\begin{proof}
Let $\{U_\alpha \}$ be a collection of open sets in $\col{T}$. Suppose $\tau \in \bigcap U_\alpha$. Then for each $U_\alpha$ there exists a $\tau_\alpha$ such that $\tau \in \st \tau_\alpha \subseteq U_\alpha$. Hence $\tau_\alpha \subseteq \tau$ and $\st \tau \subseteq \st \tau_\alpha \subseteq U_\alpha$. 
\end{proof}

\begin{proposition}
A subset of an abstract simplicial complex is closed if and only if it is a subcomplex.
\end{proposition}

\begin{proof}
Let $X$ be an abstract simplicial complex and $A \subseteq X$. Suppose $A$ is closed and $\tau \in A$. If $\sigma \subseteq \tau$ and $\sigma \in X \setminus A$ then $\tau \in \st \sigma \subseteq X \setminus A$. Hence $\sigma \in A$ and $A$ is a subcomplex of $X$. Conversely, suppose $A$ is a subcomplex of $X$ and $\sigma \in X \setminus A$. If $\sigma \subseteq \tau$ and $\tau \in A$ then $\sigma \in A$, hence $\tau \in X \setminus A$ and $\st \sigma \subseteq X \setminus A$ so $X \setminus A$ is open.
\end{proof}

\begin{definition}
  Starting with a subset $A \subseteq X$ of an abstract simplicial complex, the following are useful related subsets:
  \begin{enumerate}
  \item The \emph{closure} $\cl A$ is the smallest closed set containing $A$.
  \item The \emph{star} $\st A$ is the smallest open set containing $A$.  It is also (see \cite[p. 371]{Munkres_1984}) given by the set of all simplices that contain a simplex in $A$.
  \item The \emph{interior} $\interior A$ is the largest open set contained in $A$.
  \item The \emph{link} $\lk A$ is the set of all simplices in $\cl \st A$ whose vertex sets are disjoint from $A$ \cite[p. 371]{Munkres_1984}, or $\lk A = (\cl \st A) \backslash (\st A \cup \cl A)$. 
  \item The \emph{frontier}\footnote{The frontier is often called the \emph{boundary}, but we find that this is often confused with other senses of the word ``boundary''.} is $\fr A = \cl A \cap \cl (X \backslash A)$.
  \end{enumerate}
\end{definition}

\begin{definition}
  If $X$ and $Y$ are simplicial complexes, a function $f$ that takes vertices of $X$ to vertices of $Y$ is called an \emph{(order preserving) simplicial map} $f:X \to Y$ whenever every simplex $[v_0, \dotsc, v_n]$ of $X$ is taken to a simplex\footnote{Removing duplicate vertices as appropriate} $[f(v_0), \dotsc, f(v_n)]$. 
\end{definition}

\begin{proposition} A map $f: X \rightarrow Y$ between Alexandrov spaces is continuous if and only if it preserves the pre-orders induced by their topologies. 
\end{proposition}
\begin{proof}\cite{May_2003, barmak2011algebraic}
Suppose $f:X \rightarrow Y$ is continuous and $\sigma \leq \tau$ in $X$. Then $\st \sigma \subseteq \st \tau \subseteq f^{-1}(\st f(\tau))$ so $f(\sigma) \in \st f(\tau)$ and $f(\sigma) \leq f(\tau)$. Conversely suppose $\sigma \in X$ and $\tau \in f^{-1}(\st f(\sigma) )$. Then $f(\tau) \in \st f(\sigma)) $ and $f(\tau) \leq f(\sigma)$. Hence $\tau \leq \sigma$ so $\st \tau \subseteq \st \sigma \subseteq f^{-1}(\st \sigma)$. Since this is true for every such $\tau$, $f^{-1}(\st \sigma)$ is open.
\end{proof}
Using Lemma \ref{lem:aspace} and the fact that simplicial maps preserve subset inclusion it immediately follows from the proposition that:
\begin{corollary} Simplicial maps between abstract simplicial complexes are continuous.
\end{corollary}

\subsection{Representing data with simplicial complexes}

There are several common ways to obtain abstract simplicial complexes from data, for instance:
\begin{enumerate}
\item By triangulating a manifold or some other volume, in which case the volume is homeomorphic to the \emph{geometric realization} of an abstract simplicial complex, 
\item Constructing the Dowker complex \cite{Dowker_1952} of a relation,
\item Computing the \v{C}ech complex of a cover, or
\item Computing the Vietoris-Rips complex of a set of points in a pseudometric space.
\end{enumerate}
The Vietoris-Rips complex is based on the construction of the \emph{flag complex}, which is useful in its own right.  Datasets are often provided in the form of undirected graphs $G=(V,E)$, which correspond to $1$-dimensional abstract simplicial complexes.  The study of an undirected graph can be enhanced by enriching it into a \emph{flag complex}.

\begin{definition}
  \label{def:flag_complex}
The \emph{flag complex $F(G)$ is the abstract simplicial complex based on a graph $G$} consisting of the set of all simplices $[v_0,\dotsc,v_k]$ such that every pair of vertices giving a $1$-simplex $[v_i,v_j]$ in $F(G)$ corresponds to an edge in $G$.
\end{definition}

\begin{proposition}
A subset of vertices in a graph $G=(V,E)$ corresponds to a simplex in the flag complex based on $G$ if and only if it is a clique in $G$. 
\end{proposition}

Although it follows from the Proposition that the flag complex contains no additional information beyond what is contained in the graph, the information is sometimes better organized.  Particularly when some graph neighborhoods are denser than others, this is reflected in the Alexandrov topology of its flag complex.  Therefore, graph-theoretic properties are encapsulated as topological properties.

\subsection{Relative simplicial homology}
\label{sec:relative_homology}
Suppose that $Y \subseteq X$ is a subcomplex of an abstract simplicial complex.

\begin{definition}
  \label{def:chain_space}
The \emph{relative $k$-chain space} $C_k(X,Y)$ is the abstract vector space\footnote{Since the software presented in Section \ref{sec:computation} works over $\mathbb{R}$ vector spaces, we avoid the obvious generalization to modules over some ring.} whose basis consists of the $k$-dimensional faces of $X$ that are not in $Y$.  We also write $C_k(X)$ in place of $C_k(X,\emptyset)$.  Given these spaces, we can define the \emph{relative boundary map} $\partial_k : C_k(X,Y) \to C_{k-1}(X,Y)$ given by
\begin{equation}
  \label{eq:boundary}
  \partial_k ([v_0,\dotsc,v_k]) = \sum_{i=0}^k (-1)^i
  \begin{cases}
    [v_0,\dotsc,v_{i-1},v_{i+1},\dotsc,v_k] &\text{if } [v_0,\dotsc,v_{i-1},v_{i+1},\dotsc,v_k] \notin Y,\\
    0 & \text{otherwise}\\
\end{cases}
\end{equation}
\end{definition}
Note that the vertex ordering is preserved by deletion, so the above formula is well-defined.  We call the sign $(-1)^i$ the \emph{orientation} of the face $[v_0,\dotsc,v_{i-1},v_{i+1},\dotsc,v_k]$ within $[v_0,\dotsc,v_k]$.

\begin{proposition}(Completely standard, for instance see \cite[Lemma 2.1]{Hatcher_2002})
The sequence of linear maps $(C_\bullet(X,Y),\partial_\bullet)$ is a chain complex.
\end{proposition}

\begin{definition}
  \label{def:relative_simplicial_homology}
If $Y \subseteq X$ is a subcomplex of an abstract simplicial complex, then $H_k(X,Y)=H_k(C_\bullet(X,Y),\partial_\bullet)$ is called the \emph{relative homology of the pair $(X,Y)$}.  We usually write $H_k(X)=H_k(X,\emptyset)$, which is the \emph{simplicial homology} of $X$.  
\end{definition}

\begin{proposition}\cite[Props. 2.9, 2.19]{Hatcher_2002}
  \label{prop:relative_functor}
Each continuous function $f: X \to Z$ from one abstract simplicial complex to another which restricts to a continuous function $Y \to W$ induces a linear map $H_k(X,Y)\to H_k(Z,W)$ for each $k$.  We call $(X,Y)$ and $(Z,W)$ \emph{topological pairs} and $f$ a \emph{pair map} $(X,Y) \to (Z,W)$.
\end{proposition}

Therefore, relative homology is a functor from the category of topological pairs and pair maps to the category of vector spaces.

\begin{proposition}\cite[Thm. 2.20, Cor. 2.11]{Hatcher_2002} 
  \label{prop:homotopy_invariant}
$H_k(X)$ is homotopy invariant: a homotopy equivalence $X \to Y$ between two abstract simplicial complexes induces isomorphisms $H_k(X) \cong H_k(Y)$ for all $k \geq 0$ . 
\end{proposition}

\subsection{Local homology}
\label{sec:local_homology_def}

\begin{definition}(compare \cite[end of Sec. 2.1]{Hatcher_2002}, \cite{MilnorStasheff_1974})
  \label{def:local_homology}
  For an open subset $U \subseteq X$ of an abstract simplicial complex, the \emph{local homology} at $U$ is $H_k(X,X \backslash U)$.  For brevity, we usually write
  \begin{equation*}
    \beta_k (U) = \dim H_k(X,X \backslash U)
  \end{equation*}
  for the \emph{local $k$-Betti number at $U$}.
\end{definition}
 
\begin{proposition} (Excision for abstract simplicial complexes, compare \cite[Lem. 35.1]{Munkres_1984})
  \label{prop:local_excision}
If $U$ is an open set of an abstract simplicial complex $X$, then $H_k(X,X \backslash U)\cong H_k(\cl U, \fr U)$.
\end{proposition}

Because we assume that abstract simplicial complexes are locally finite, Proposition \ref{prop:local_excision} indicates that local homology can be computed using finite dimensional linear algebra provided the open set $U$ is finite.

\begin{proof}
  It suffices to show that the chain complexes associated to $H_k(X,X \backslash U)$ and $H_k(\cl U, \fr U)$ are exactly the same.  In both cases, the chain spaces $C_k(X,X\backslash U)$ and $C_k(\cl U, \fr U)$ both consist of a vector space whose basis is the set of simplices in $U$.  Also observe that since $\cl U$ is a closed subcomplex of $X$, the boundary map $\partial_k: C_k(X) \to C_{k-1}(X)$ restricts\footnote{This does not occur for general topological spaces!} to a map $\partial_k: C_k(\cl U) \to C_{k-1}(\cl U)$. Likewise, since $\fr U$ is a closed subcomplex of $X \backslash U$, the boundary map $\partial_k: C_k(X\backslash U)\to C_{k-1}(X \backslash U)$ restricts to a map $\partial_k: C_k(\fr U) \to C_{k+1}(\fr U)$.  Collecting these facts, we conclude that $\partial_k: C_k(X,X \backslash U) \to C_{k-1}(X,\backslash U)$ restricts to $\partial_{k} : C_k(\cl U,\fr U) \to C_{k-1}(\cl U, \fr U)$.  But, we previously established that the domains and ranges of these maps are identical, so the maps must in fact be identical.  Having shown that the chain complexes are identical, we conclude that their homologies must also be identical.
\end{proof}

As an aside, we note that this is somewhat stronger than the usual excision principle, a usual formulation of which reads:
\begin{proposition} (Excision principle, \cite[Thm. 2.20]{Hatcher_2002})
  \label{prop:usual_excision}
If $U$ and $V$ are sets in a topological space $X$ for which $\cl V \subseteq X \backslash (\cl U)$, then $H_k(X,X \backslash U)\cong H_k(X \backslash V,X \backslash (U\cup V))$.
\end{proposition}
We could attempt to derive Proposition \ref{prop:local_excision} from Proposition \ref{prop:usual_excision} by taking an open set $U$ and $V = \interior X\backslash U$, but this choice of $V$ violates the hypotheses of Proposition \ref{prop:usual_excision} because $\cl V = X \backslash U$ which is not generally a subset of $X\backslash (\cl U)$.

\begin{proposition} (see also \cite{Brown_2017} for a similarly elementary proof)
\label{prop:homology_sheaf}
The functor $U \mapsto H_k(X,X \backslash U)$ defines a sheaf $\shf{H}_k$; called the \emph{$k$-homology sheaf}.
\end{proposition}
\begin{proof}

  \emph{Restriction maps:} Suppose $\sigma \subseteq \tau$ are two faces of $X$.  Then
  \begin{eqnarray*}
    \sigma & \subseteq & \tau \\
    \st \sigma & \supseteq & \st \tau\\
    X \backslash \st \sigma & \subseteq & X \backslash \st \tau
  \end{eqnarray*}
  so there is an inclusion of topological pairs
  \begin{equation*}
  (X,X \backslash \st \sigma) \hookrightarrow (X,X \backslash \st \tau)
  \end{equation*}
  which induces linear maps (Proposition \ref{prop:relative_functor} in Section \ref{sec:relative_homology})
  \begin{equation*}
    H_k(X,X \backslash \st \sigma) \to H_k(X,X \backslash \st \tau),
  \end{equation*}
  one for each $k$.  These linear maps form the restriction maps for the sheaf since the topology on $X$ is generated by the stars over faces.  

  \emph{Monopresheaf:} (compare \cite{MilnorStasheff_1974}; it is much harder to show that one obtains a sheaf of local \emph{singular} homology) To show the uniqueness of gluings it is sufficient to show that the restriction maps are injective. Suppose we have $[z]\in H_k(X,X\backslash U)$ with $[z] \not= 0$ for some relative cycle $z$.  Observe that $z$ is a linear combination of simplices in $U$.  By assumption, at least one coefficient in this linear combination is nonzero.  Consider a simplex $\sigma \in U$ whose coefficient in $z$ is nonzero.  Then under the induced map $H_k(X,X\backslash U) \to H_k(X,X\backslash \st \sigma)$, this coefficient remains unchanged and thus remains nonzero. 
     
  \emph{Conjunctive:} Suppose that there are two classes $[x]\in H_k(X,X \backslash U)$, $[y]\in H_k(X,X \backslash V)$ whose restrictions to $U\cap V$ are equal; we must show that there exists a $[z]\in H_k(X,X \backslash (U\cup V))$ that restricts to $[x]$ on $U$ and $[y]$ on $V$.  Using the appropriate inclusions of topological pairs, we can set up a short exact sequence of relative $k$-chains
  \begin{equation*}
    0 \to C_k(X,X \backslash (U \cup V)) \to C_k(X,X \backslash U) \oplus C_k(X,X \backslash V) \to C_k(X,X \backslash (U \cap V)) \to 0
  \end{equation*}
  where the map to $C_k(X,X \backslash (U \cap V))$ computes the difference of the two restrictions.  The resulting long exact sequence
  \begin{equation*}
    \dotsb \to H_k(X,X \backslash (U \cup V)) \to H_k(X,X \backslash U) \oplus H_k(X,X \backslash V) \to H_k(X,X \backslash (U \cap V)) \to \dotsb
  \end{equation*}
  does precisely what we want: since $([x],[y])$ in the middle space lies in the kernel of the map to $H_k(X,X \backslash (U \cap V))$, it must be the image of some $[z]\in H_k(X,X \backslash (U\cup V))$.
\end{proof}

\begin{corollary}
\label{cor:lh_sections}
Global sections of the homology sheaf are the homology classes of the abstract simplicial complex.  On the other hand, reduced homology classes are obtained as local sections over sets of the form $X \backslash \{\sigma\}$ for any $x\in X$.
\end{corollary}
\begin{proof}
Given Proposition \ref{prop:homology_sheaf}, we need only compute 
\begin{eqnarray*}
\shf{H}_k(X) &=& H_k(X,X \backslash X) \\ 
&=& H_k(X,\emptyset) \\ &=& H_k(X).
\end{eqnarray*}
Nearly the same calculation works for reduced homology classes, yielding
\begin{equation}
\shf{H}_k(X\backslash \{\sigma\}) = H_k(X,\{\sigma\})
\end{equation}
which is reduced homology (see \cite[Ex. 2.18]{Hatcher_2002}).
\end{proof}

\begin{remark}
If $X$ is an abstract simplicial complex whose geometric realization is an $n$-dimensional manifold, then the $n$-homology sheaf is usually called the \emph{orientation sheaf} \cite{Kashiwara_1990,Milnor_1963}.
\end{remark}

As we'll see in Section \ref{sec:stratification_detection}, local homology detects stratifications of triangulated manifolds.  From a practical standpoint this can be difficult to apply using sampled data.  The following elegant proposition shows that the local homology of a sufficiently nice topological space is described by the local homology of abstract simplicial complexes.

\begin{proposition}(Not explicitly stated as a theorem, but proven in Section 7 of \cite{Bendich_2007})
  \label{prop:recovery}
  Let $X$ be a locally compact subspace of $\mathbb{R}^n$, $\col{U}$ be an open cover of $X$, and $U_\alpha$ be an $\alpha$-offset of an open subset $U \in \col{U}$.  Then
  \begin{equation*}
    H_k(U_\alpha \cap B_r, U_\alpha \cap \partial B_r) \cong H_k(K,K_0),
  \end{equation*}
  where $B_r$ is a radius $r>0$ ball around some point in $U_\alpha$, and $K$ is a particular abstract simplicial complex (the nerve of $\col{U}$) and $K_0$ is a subcomplex of $K$.
\end{proposition}

\begin{figure}
\begin{center}
\includegraphics[width=4in]{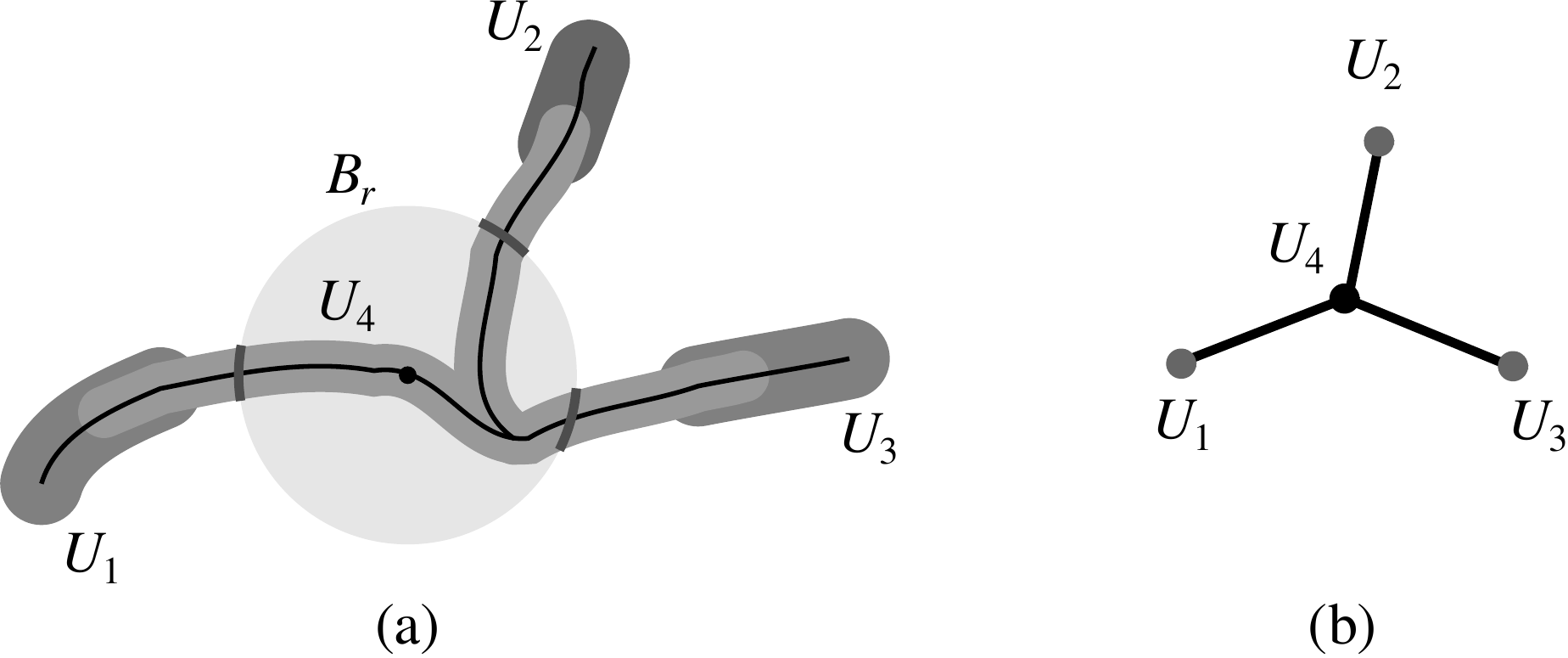}
\caption{A locally compact subspace of $\mathbb{R}^2$ covered by four open sets (a) and its simplicial complex model (b)}
\label{fig:recovery_example}
\end{center}
\end{figure}

\begin{example}
  \label{eg:recovery_example}
  Consider the space shown in Figure \ref{fig:recovery_example}(a), which is covered by four open sets, $U_1$, $U_2$, $U_3$, and $U_4$.  In that figure, consider the open ball $B_r$ centered on a point in the space and the intersection of $B_r$ and $U_r$.  Notice that although the intersection $U_4 \cap \partial B_r$ is not open, it deformation retracts to the open intersection $U_4 \cap (U_1 \cup U_2 \cup U_3)$.  The nerve $K$ of the open cover is shown in Figure \ref{fig:recovery_example}(b), which consists of four vertices and three edges.  Of the vertices, the local homology at the vertex for $U_4$ is a model for the local homology near the branch point of the space contained in $B_r$.  Specifically, in the nerve, the three vertices corresponding to $U_1$, $U_2$, and $U_3$ form the complex $K_0$ in Proposition \ref{prop:recovery}.
\end{example}

\section{Local homology of graphs}
\label{sec:graphs}

Let $G=(V,E)$ be an undirected graph with vertex set $V$ and edge set $E$.  We recognize immediately there is a bijective map between $G$ and a $1$-dimensional abstract simplicial complex with $0$-simplices corresponding to $V$ and $1$-simplices corresponding to $E$. We will freely move between the context of $G$ and its associated abstract simplicial complex in the following discussion using $v$ to represent a vertex and $[v]$ its corresponding $0$-simplex.

It is useful to define the open and closed neighborhoods for a vertex $v \in V$ to relate the combinatorial graph structure to the corresponding topology around $\st [v]$. To prevent equivocation we will use the word \emph{neighborhood} within the context of a graph and open set or star within the context of the topology of the associated abstract simplicial complex.

\begin{definition}
  \label{def:graph_neighborhood}
The \emph{open neighborhood\footnote{Sometimes in the literature $N(v)$ and $\overline{N}(v)$ are used to denote
only set of vertices in the neighborhood.} of a vertex $v \in V$} is the subgraph of $G$
induced by all neighboring vertices of $v$
\begin{equation*}
  N(v) = G\left[ \{w \in V : (v,w) \in E \} \right],
\end{equation*}
where the notation $G[ W ]$ indicates the graph induced by a set $W$ of vertices. 
Note that $v$ is not in $N(v)$ since $G$ is a simple graph.   Moreover, $N(v)$ does not include any edge incident to $v$.  We include $v$ and the edges incident to it in the \emph{closed neighborhood of $v$} which is defined as
\begin{equation*}
  \overline{N}(v) = G\left[ \{v\} \cup \{w \in V : (v,w) \in E \} \right].
\end{equation*}
\end{definition}

\begin{proposition}
In the flag complex $F(G)$, the neighborhood $N(v)$ of a vertex $v\in G$ corresponds to the 1-skeleton of the link $\lk [v]$ and $\overline{N}(v)$ corresponds to the 1-skeleton of $\cl \st [v]$.
\end{proposition}
\begin{proof}
Observe that the set 
\begin{equation*}
\{w \in V : (v,w) \in E \}
\end{equation*}
is the set of vertices in 
\begin{eqnarray*}
  \left(\cl \st [v]\right) \setminus \st [v] &=& \fr \st [v]\\
  &=& (\cl \st [v]) \cap \cl(X \backslash \st [v])\\
  &=& (\cl \st [v]) \cap (X \backslash \st [v])\\
  &=& (\cl \st [v]) \backslash (\st [v])\\
  &=& (\cl \st [v]) \backslash (\st [v] \cup \cl [v])\\
  &=& \lk [v]. 
\end{eqnarray*}

(We relied on the fact that $\st [v]$ is open (third line) and that $[v]$ is closed (fifth line).)
Every edge in $N(v)$ is in $\lk v$ because $\lk [v]$ is an abstract simplicial complex, and for the same reason every edge in $\lk [v]$ is also in $N(v)$.  Including the vertex $v$ to form $\overline{N}(v)$ yields the vertex set of $\cl \st [v]$, so this completes the proof.  
\end{proof}

\begin{definition}
  The number of connected components in the open neighborhood of $v$ will be important in our results. We will denote that by $|\pi_0(N(v))|$ or simply $\pi_0$ when there is no confusion on the vertex choice. We will denote the degree of a vertex $v$ as
\begin{equation*}
  \deg v = |\{w \in V : (v,w) \in E \}|.
\end{equation*}
\end{definition}

\begin{proposition}
  \label{prop:graph_degree}
If $X$ is the $1$-dimensional abstract simplicial complex corresponding to a graph $G(V,E)$ then 
\begin{equation*}
  \beta_1(\st [v] ) = \dim H_1(X,X \backslash \st [v]) = \deg v - 1
\end{equation*}
for each vertex $v$.
\end{proposition}

We will later take $1 + \beta_1(\st \sigma)$ to be the \emph{generalized degree} of a simplex $\sigma$ in an abstract simplicial complex.  It is also useful to compare Proposition \ref{prop:graph_degree} with Theorem \ref{thm:bound} in Section \ref{sec:general} which makes a more general and more global statement, but is less tight.

\begin{proof} 

  By excision (Proposition \ref{prop:local_excision} in Section \ref{sec:local_homology_def}), we have that
  \begin{eqnarray*}
    \beta_1(\st [v] ) &=& \dim H_1(X,X \backslash \st [v])\\
    &=& \dim H_1(\cl \st [v], \fr \st [v]).
  \end{eqnarray*}
  Because $X$ is a  $1$-dimensional abstract simplicial complex, $\fr \st [v] = \cl (\st [v]) \cap \cl (X \backslash \st [v])$
  contains no edges, and therefore contains only vertices.  (Suppose $\epsilon$ were to be an edge in $\fr \st [v]$.  Since the closure of any subset $A \subseteq X$ differs from $A$ only in its vertices, then $\epsilon \in \st [v] \cap (X \backslash \st [v]) = \emptyset$, which is a contradiction.)  The number of vertices in $\fr \st [v]$ is precisely the degree of $v$.  

  Therefore, the long exact sequence for the pair $(\cl \st [v], \fr \st [v])$ is
  \begin{equation*}
      0 \to H_1(\cl \st [v]) \to H_1(\cl \st [v],\fr \st [v]) \to \mathbb{R}^{\deg v} \xrightarrow{i} \mathbb{R} \to H_0(\cl \st [v],\fr \st [v]) \to 0 
  \end{equation*}
  
  The map labeled $i$ above represents the map induced by the inclusion
  \begin{equation*}
    \fr\st [v] \hookrightarrow \cl\st [v],
  \end{equation*}
  and so is surjective.
  
  We claim that $H_1(\cl \st [v])=0$, because this means that $H_1(\cl \st [v],\fr \st [v])$ injects into $\mathbb{R}^{\deg v}$.  Because of the surjectivity of $i$, this means that the dimension of $H_1(\cl \st [v],\fr \st [v])$ must be $(\deg v) - 1$ as the theorem states.
  
  To prove the claim on $H_1(\cl \st [v])$, consider the chain complex:
  \begin{equation*}
    \xymatrix{
      0 \ar[r] & \mathbb{R}^{\deg v} \ar[r]^{\partial_1} & \mathbb{R}^{1+\deg v} \ar[r] & 0 \\
    }
  \end{equation*}
  in which the boundary map is given by
  \begin{equation*}
    \partial_1 = \begin{pmatrix}
      1  &  1 & \dotsb & 1\\
      -1 &  0 & \dotsb & 0\\
      0  & -1 & \dotsb & 0\\
         &    & \vdots &  \\
      0  &  0 & \dotsb & -1\\   
      \end{pmatrix}.
  \end{equation*}
  in which the first row corresponds to $[v]$ and each column corresponds to an edge incident to $[v]$.  Clearly $\partial_1$ is injective, so $H_1(\cl \st [v])=0$.
\end{proof}

\subsection{Basic graph definitions}

\begin{definition}
A graph $G=(V,E)$ is \emph{planar} if it can be embedded in the plane such that
edges only intersect at their endpoints. In other words, it can be drawn so
that no edges cross. This drawing is called a \emph{planar embedding}. An
example is shown in Figure \ref{fig:planar}.  Each simply-connected region in the plane bounded by the embedding of a cycle in the graph is called a \emph{face}.  Faces are \emph{bounded} (or \emph{unbounded}) if they are compact (or not compact, respectively) in the plane.  The set of faces is denoted $A$.
\end{definition}
One important property of planar graphs, for the purposes of this article,
is that they are \emph{locally outerplanar}. In other words, for any $v \in V$,
$N(v)$ is an outerplanar graph. This will be used to prove our bounds on the
local clustering coefficient.
\begin{definition}
A graph is \emph{outerplanar} if it has a planar embedding such that all
vertices belong to the unbounded face of the drawing. An example is shown in
Figure \ref{fig:outerplanar}.
\end{definition}
\begin{figure}[ht]
\begin{center}
  \begin{subfigure}[b]{0.45\textwidth}
  \begin{center}
    \includegraphics{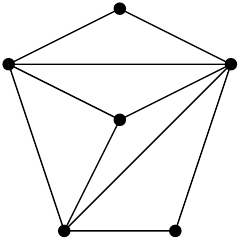}
    \caption{A planar embedding of a graph.}\label{fig:planar}
  \end{center}
  \end{subfigure}
  \begin{subfigure}[b]{0.45\textwidth}
  \begin{center}
    \includegraphics{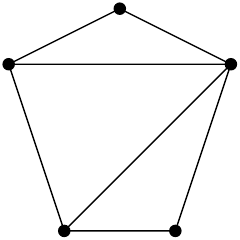}
    \caption{An outerplanar graph.}\label{fig:outerplanar}
  \end{center}
  \end{subfigure}
  \caption{}
\end{center}
\end{figure}

A minimal outerplanar graph on $n$ vertices is a tree, and thus has
$n-1$ edges. A maximal outerplanar graph is a triangulation and has $2n-3$
edges (this can be derived by using the Euler characteristic formula for planar graphs, $|V|-|E|+|A|=2$, and observing that each edge is contained in two faces while each finite face is bounded by 3 edges and the infinite face is bounded by $|V|$ edges). The graph in Figure \ref{fig:planar} cannot be outerplanar (that is,
it cannot be drawn with all vertices on the unbounded face) because it has 6
vertices and 10 edges, and an outerplanar graph on 6 vertices has at most
$2 \cdot 6 - 3 = 9$ edges.

\subsection{Local homology at a vertex}
\label{sec:graph_results}

The local homology of a vertex in a
graph considers the structure of the neighborhood of a single vertex in
relation to the rest of the graph. In this section we work in the context of the flag
complex $F(G)$ rather than the graph $G$ itself. Note that this is different than the perspective taken in Proposition \ref{prop:graph_degree} which treats the graph as a simplicial complex on its own rather than through the lens of its flag complex.

Local homology can be defined with respect to any open subset of an abstract simplicial complex
$X$ using Definition \ref{def:local_homology}. In the notation of this definition let $X = F(G)$ and $U = \st\cl(L)$ for some $L \subseteq X$. We are specifically interested
in the case where $L = \{v\}$ for some $v \in V$. In this case, $\st\cl(L)$ will
consist of $\{v\}$ itself, all edges incident to $v$, and an $(i+1)$-simplex
for every $K_{i}$ in $N(v)$.

\begin{figure}
\begin{center}
\includegraphics[width=4in]{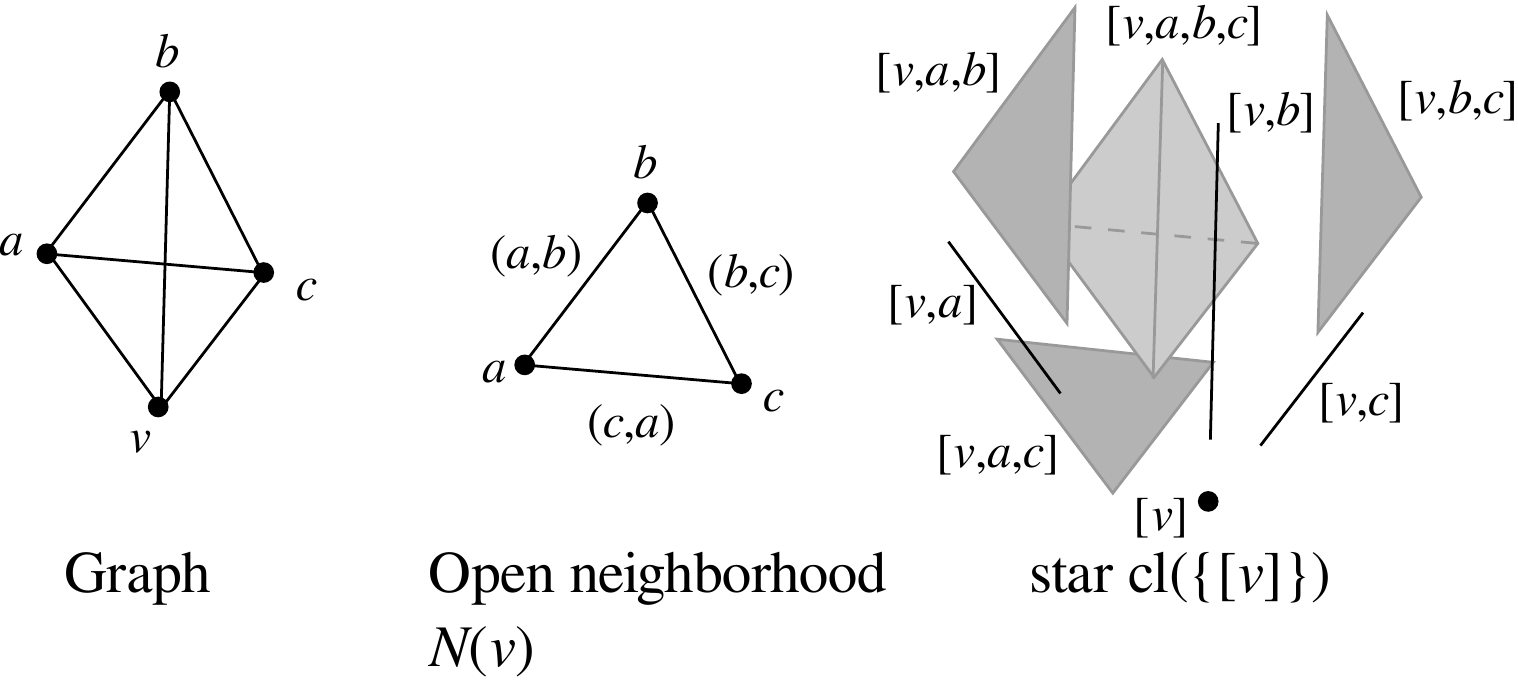}
\caption{A graph containing vertices $\{a,b,c,v\}$ (left), the open neighborhood $N(v)$ (center), and the star of the closure of $v$ in the flag complex of the graph (right)}
\label{fig:starcl_graph}
\end{center}
\end{figure}

\begin{example}
  \label{eg:starcl_graph}
  Consider the graph shown in Figure \ref{fig:starcl_graph}.  In this graph, $N(v)$ contains edges $\{(a, b), (b,c), (c,a)\}$ forming a $K_3$.  Then $\st\cl(\{[v]\})$ contains vertex $[v]$, edges $\{[v, a], [v,b], [v,c]\}$, 3-simplices $\{[v,a,b], [v,a,c], [v,b,c]\}$, and 4-simplex $[v, a, b, c]$.
\end{example}

For any $0$-simplex $[v]$ in $F(G)$, the first local Betti number $\beta_1(\st [v])$ is computed using the rank-nullity theorem on the linear maps between the chain spaces $C_k(F(G),F(G) \backslash (\st [v]))$. Specifically:
\begin{equation*}
  \beta_1(\st[v]) = \dim(C_1(F(G), F(G)\setminus \st
  \cl [v])) - \dim(\text{im}(\partial_1)) - \dim(\text{im}(\partial_{2})).
\end{equation*}

Now with all of the definitions out of the way we can
present our two bound results for planar graphs. First we bound the local
clustering coefficient of a vertex $v$ in terms of its degree, $\deg v$, and the
number of connected components in its open neighborhood, $\pi_0$.
\begin{lemma}\label{lem:CC_bounds}
For a planar graph $G=(V,E)$ and vertex $v \in V$ with degree $\deg v$ and $\pi_0$
connected components in $N(v)$ we have
\begin{equation*}
  \frac{2(\deg v-\pi_0)}{\deg v(\deg v-1)} \leq CC(v) \leq \frac{6(\deg v-\pi_0)}{\deg v(\deg v-1)}.
\end{equation*}
\end{lemma}
\begin{proof}
Let $N(v) = \{N(v)_i\}_{i=1}^{\pi_0}$ be the partition of $N(v)$ into its connected
components. Let $n_i := |V(N(v)_i)|$ be the number of vertices in each
connected component so that $\deg v = \sum_{i=1}^{\pi_0} n_i$. Additionally, let $m_i
:= |E(N(v)_i)|$ be the number of edges in each connected component, with
$|E(N(v))| = \sum_{i=1}^{\pi_0} m_i$. Since $G$ is planar, $N(v)$ must be
outerplanar. Moreover, each $N(v)_i$ must be outerplanar. Therefore we can
use the bounds on the number of edges in an outerplanar graph to bound the
clustering coefficient. We will start with the lower bound.
\begin{align*}
CC(v) &= \frac{|E(N(v))|}{\binom{\deg v}{2}}\\
      &= \frac{2\cdot \sum_{i=1}^{\pi_0} m_i}{\deg v(\deg v-1)}\\
      &\geq \frac{2\cdot \sum_{i=1}^{\pi_0} (n_i-1)}{\deg v(\deg v-1)}\\
      &= \frac{2(\deg v-{\pi_0})}{\deg v(\deg v-1)}.
\end{align*}
Now, for the upper bound, notice that if $n_i > 1$ the bound of $m_i \leq
2n_i-3$ makes sense. But if $n_i=1$ then $m_i=0$ and not $2n_i-3=-1$.
Therefore, when bounding $|E(N(V))|$ from above we must take this into
account.
\begin{align*}
CC(v) &= \frac{|E(N(v))|}{\binom{\deg v}{2}}\\
      &= \frac{2\cdot \sum_{i=1}^{\pi_0} m_i}{\deg v(\deg v-1)}\\
      & \leq \frac{2 \left[\sum_{i=1}^{\pi_0} (2 n_i-3)+(\text{number of singleton }N(v)_i)\right]}{\deg v(\deg v-1)}
\end{align*}
We must add this ``number of singleton $N(v)_i$'' because for every singleton
$N(v)_i$ we have a -1 contribution from $2n_i-3$. This is counteracted by
adding +1 for each of these singleton components. Letting this number of
singleton components equal $s_v$ we may finishing the upper bound.
\begin{align*}
CC(v) &\leq \frac{2 ( 2\deg v-3\pi_0+s_v )}{\deg v(\deg v-1)}\\
      &\leq \frac{2 (2\deg v-3\pi_0+\deg v)}{\deg v(\deg v-1)}\\
      &= \frac{6 (\deg v-\pi_0)}{\deg v(\deg v-1)}
\end{align*}
\end{proof}

Next, we will use these bounds along with Proposition \ref{prop:graph_degree} to
establish a functional relationship between $\beta_1(\st [v])$ and $CC(v)$.
\begin{theorem}\label{thm:LH_CC_bounds}
Let $G=(V,E)$ be a planar graph, and $v \in V$. Then we may bound
$\beta_1(\st [v])$ with functions of $CC(v)$ and $\deg v$
\begin{equation*}
  \deg v-1-\frac{\deg v(\deg v-1)CC(v)}{2} \leq \beta_1(\st [v]) \leq \deg v-1-\frac{\deg v(\deg v-1)CC(v)}{6}.
\end{equation*}
\end{theorem}
\begin{proof}
For the proof of this Theorem we will use shorthand and denote $H_v :=
\beta_1(\st [v])$. From Theorem \ref{thm:bound} (proved in Section \ref{sec:general}) we know that the number
of connected components in $N(V)$ can be written in terms of the dimension 1
relative homology at vertex $v$
\begin{equation*}
  \pi_0 = H_v+1.
\end{equation*}
Then, the upper bound for $CC(v)$ in terms of $\deg v$ and $\pi_0$ from Lemma
\ref{lem:CC_bounds} can be turned into an upper bound on $H_v$ in terms of
$CC(v)$.
\begin{align*}
  CC(v) &\leq \frac{6(\deg v-(H_v+1))}{\deg v(\deg v-1)}\\
  \frac{\deg v(\deg v-1)CC(v)}{6} &\leq \deg v-H_v-1\\
  H_v &\leq \deg v-1-\frac{\deg v(\deg v-1)CC(v)}{6}
\end{align*}
Similarly the lower bound from \ref{lem:CC_bounds} can be turned into a lower
bound for $H_v$.
\begin{align*}
  CC(v) &\geq \frac{2(\deg v-(H_v+1))}{\deg v(\deg v-1)}\\
  \frac{\deg v(\deg v-1)CC(v)}{2} &\geq \deg v-H_v-1\\
  H_v &\geq \deg v-1-\frac{\deg v(\deg v-1)CC(v)}{2}
\end{align*}
\end{proof}

\section{Local homology of general complexes}
\label{sec:general}

Generalizing Proposition \ref{prop:graph_degree} from Section \ref{sec:graphs} to all abstract simplicial complexes provides a generalization of the degree of a vertex to all simplices.  This quantity also has a convenient interpretation in terms of connected components.

\begin{theorem}
\label{thm:bound}\cite{RobinsonGlobalSIP2014}
Suppose that $X$ is an abstract simplicial complex, and that $\sigma$ is a face of $X$.  If $X$ is connected and $\st \sigma$ is a proper subset of $X$, then $\beta_1(\st \sigma) + 1$ is an upper bound on the number of connected components of $X \backslash \st \sigma$.  When $H_1(X)$ is trivial, that upper bound is attained.
\end{theorem}

The proof is a short computation using the long exact sequence for the pair $(X,X\,\backslash\,\st \sigma)$.

\begin{proof}
  Let $Y=X\,\backslash\, \st \sigma$.  Consider the long exact sequence associated to the pair $(X,Y)$, which is
\begin{equation*}
\xymatrix{
\dotsb \ar[r]& H_1(Y) \ar[r] & H_1(X) \ar[r] & H_1(X,Y) \ar[llld] \\
 H_0(Y) \ar[r]& H_0(X) \ar[r]& H_0(X,Y) \ar[r] & 0.
}
\end{equation*}
$H_k(X,Y) \cong \tilde{H}_k(X/Y)$ follows via \cite[Prop. 2.22]{Hatcher_2002}, where $\tilde{H}_k$ is reduced homology.  The standard interpretation of reduced homology means that $H_0(X,Y)=0$ and $H_0(X)\cong \mathbb{Z}$ because $X$ is connected.  Thus the long exact sequence reduces to
\begin{equation*}
\dotsb \to H_1(Y) \to H_1(X) \to H_1(X,Y) \overset{f}{\to} H_0(Y) \overset{g}{\to} \mathbb{Z} \to 0.
\end{equation*}
The number of connected components of $Y$ is $\rank H_0(Y)$.  Because of the last term in the long exact sequence, this is at least 1.  If $H_1(X)=0$ then $H_1(Y)=0$ also, so the kernel of the homomorphism $g$ is precisely the image of the monomorphism $f$.  Hence $\rank H_1(X,Y) + 1 = \rank H_0(Y)$ as claimed.

On the other hand, if $H_1(X)$ is not trivial, then $H_1(Y)$ may or may not be trivial, depending on exactly where $a$ happens to fall.  By exactness,
\begin{eqnarray*}
  \text{dim}(\ker g) &=& \rank H_0(Y) - 1 \\
  &=& \text{dim}(\image f) \\
  &=& \rank H_1(X,Y) - \text{dim}(\ker f)\\
\end{eqnarray*}

where the rank-nullity theorem for finitely generated abelian groups applies in the last equality. The kernel of $f$ may be as large as $\rank H_1(X)$, but it may be smaller.  Thus, we can claim only that
\begin{equation*}
\rank H_0(Y) \le \rank H_1(X,Y) + 1.
\end{equation*}
\end{proof}

\begin{definition}
  We call the number $1 + \beta_1(\st \sigma)$ the \emph{generalized degree} of a simplex $\sigma$ in an abstract simplicial complex.
\end{definition}

\begin{figure}
\begin{center}
\includegraphics[width=5in]{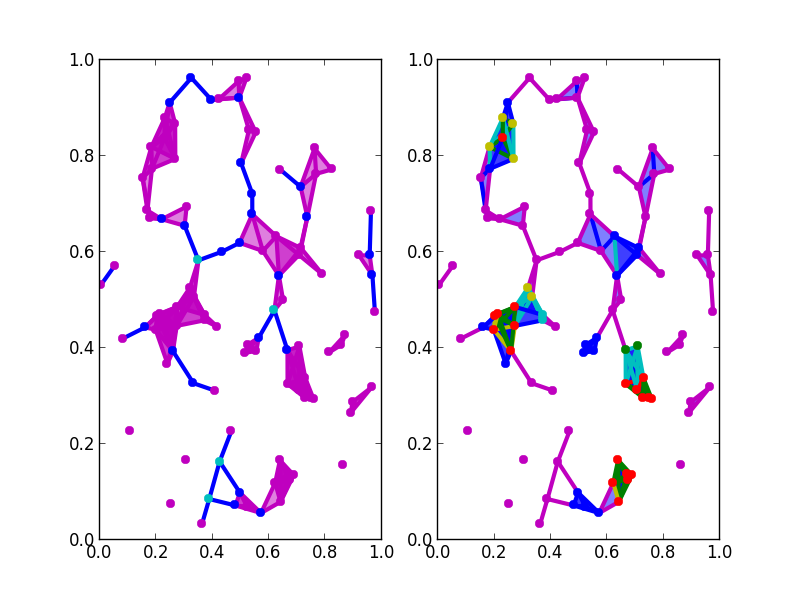}
\caption{$\beta_1$ (left) and $\beta_2$ (right) of a random simplicial complex.  Magenta: $\beta_k = 0$, Blue: $\beta_k = 1$, Cyan: $\beta_k = 2$, Green: $\beta_k = 3$, Yellow: $\beta_k = 4$, Red: $\beta_k = 5$.}
\label{fig:thick_graph}
\end{center}
\end{figure}

Figure \ref{fig:thick_graph} shows an example of a random simplicial complex that has been colored by $\beta_1$ (left) and $\beta_2$ (right), which provides some insight into why $1+\beta_1$ is a generalized degree.  Due to Proposition \ref{prop:graph_degree}, it is clear that $1+\beta_1$ reduces to the degree of a vertex in a graph.  However, for simplicial complexes Theorem \ref{thm:bound} indicates that the generalized degree has a clear topological meaning: it is the number of local connected components that remain after removing that simplex.  

\subsection{Stratification detection} 
\label{sec:stratification_detection}
Roughly speaking, a space is a manifold whenever it is locally Euclidean at each point.  Although local homeomorphisms can be difficult to construct, local homology can identify some non-manifold spaces. Recall that \emph{singular homology} is defined for \emph{all} topological spaces by studying classes of continuous maps from the standard $k$-simplices, while simplicial homology  is defined for abstract simplicial complexes (Definition \ref{def:relative_simplicial_homology}).

\begin{proposition}\cite[Thm 2.27]{Hatcher_2002}
\label{prop:triangulation}
If $Z$ is a triangulation of a topological space $X$ for which a subcomplex $W \subseteq Z$ is a triangulation of a subspace $Y\subseteq X$, then $H_k(X,Y) \cong H_k(Z,W)$ where the left side is relative singular homology and the right side is relative simplicial homology.
\end{proposition}

 Because of this proposition, we shall generally ignore the distinction between a topological space (usually a stratified manifold) and its triangulations. Moreover as a consequence of \cite{mccord1966} we associate to any triangulation an abstract simplicial complex which has the same homology groups as the triangulation. For this reason we will freely reference the homology groups of the manifold and its triangulation with the corresponding abstract simplicial complex. 

\begin{definition} \cite[pg. 198]{Munkres_1984}
An abstract simplicial complex $X$ is called a \emph{homology $n$-manifold} if 
\begin{equation*}
\dim H_k(X,X \backslash \st \sigma) = \begin{cases}
1 & \text{if }k=n\\
0 & \text{otherwise}\\
\end{cases}
\end{equation*}
for each simplex $\sigma$.  Any simplex for which the above equation does \emph{not} hold is said to be a \emph{ramification simplex}.  
\end{definition}

The presence of ramification simplices implies that a simplicial complex cannot be the triangulation of a manifold.  All ramification simplices necessarily occur along lower-dimensional strata of the triangulation of a stratified manifold, but not all strata contain ramification simplices in a triangulated stratified manifold.

Since Proposition \ref{prop:local_excision} in Section \ref{sec:local_homology_def} provides a characterization of the behavior of local homology at various simplices, this means that ramification simplices are easily detectable.

\begin{example}
  Consider again the random simplicial complex shown in Figure \ref{fig:thick_graph}.  In the portions of the complex where it appears ``graph edge-like'', for instance each edge that is not a face of any other simplex, $\beta_1=1$ and $\beta_2 = 0$.  On the other hand, places where the complex appears to be ``thickened vertices'' have nonzero $\beta_2$, indicating that they are ramification simplices.
\end{example}

Using this definition and Proposition \ref{prop:triangulation} in Section \ref{sec:relative_homology}, we obtain the following useful characterization of the local homology of triangulations.

\begin{proposition}
\label{prop:homology_manifold}
If $X$ is the triangulation of a topological $n$-manifold, then $X$ is a homology $n$-manifold.
\end{proposition}

Manifold boundaries can also be detected by their distinctive local homology.

\begin{proposition}
  \label{prop:manifold_boundary}
  If $X$ is the triangulation of a topological $n$-manifold with boundary and $\sigma$ is a simplex on that manifold boundary, then
  \begin{equation*}
    \dim H_k(X,X \backslash \st \sigma) = 0
  \end{equation*}
  for all $k$.
\end{proposition}

\begin{proof}
  Using Propositions \ref{prop:triangulation} and \ref{prop:homotopy_invariant} in Section \ref{sec:relative_homology}, without loss of generality, we consider the case of the half space
  \begin{equation*}
    \mathbb{H}^n = \{(x_1, \dotsc, x_n) \in \mathbb{R}^n : x_1 \ge 0\}
  \end{equation*}
  and compute local homology at the origin using singular homology
  \begin{equation*}
    H_k(\mathbb{H}^n, \mathbb{H}^n \backslash \{0\}).
  \end{equation*}
  If $B_\epsilon(0)$ is the open ball of radius $\epsilon> 0$ about the origin, then $\mathbb{H}^n \backslash\{0\}$ deformation retracts to $\mathbb{H}^n \backslash B_\epsilon(0)$, so
  \begin{eqnarray*}
    H_k(\mathbb{H}^n, \mathbb{H}^n \backslash \{0\}) &\cong& \widetilde{H}_k(\mathbb{H}^n/ ( \mathbb{H}^n \backslash B_\epsilon(0)))\\
    &\cong& 0
  \end{eqnarray*}
  since the quotient is contractible.
\end{proof}

\begin{figure}
\begin{center}
\includegraphics[width=3in]{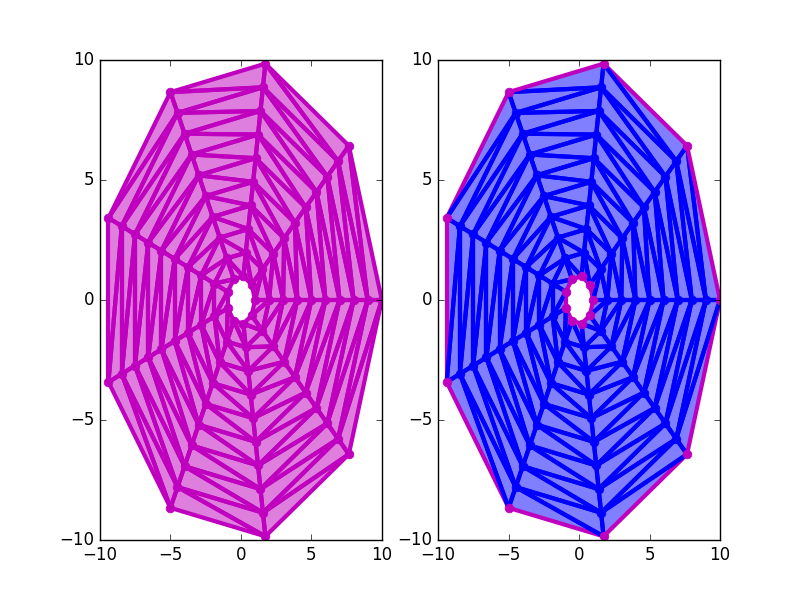}
\caption{Local Betti 1 (left) and local Betti 2 (right) of an annulus. Magenta: $\beta_k = 0$, Blue: $\beta_k = 1$.}
\label{fig:annulus}
\end{center}
\end{figure}

\begin{example}
As an example of the local homology of a manifold with boundary, Figure \ref{fig:annulus} shows $\beta_1$ and $\beta_2$ computed over all simplices in the triangulation of an annulus.  Because of Proposition \ref{prop:homology_manifold}, the local 1-homology is completely trivial over the entire space.  Since the space is locally homeomorphic to $\mathbb{R}^2$ away from its boundary, the local 2-homology has dimension 1 there.  Along the boundary, the local 2-homology is trivial in accordance with Proposition \ref{prop:manifold_boundary}.
\end{example}

\begin{figure}
\begin{center}
\includegraphics[width=4in]{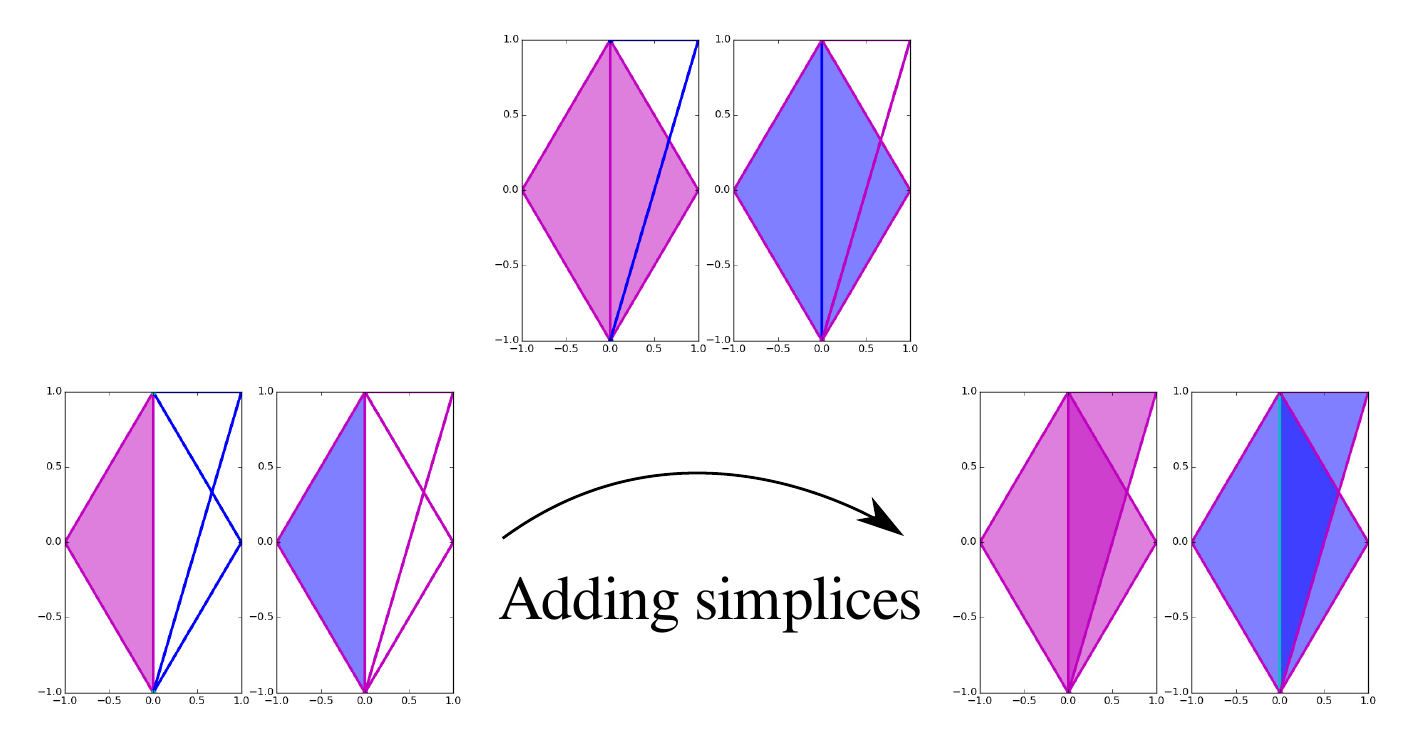}
\caption{Local Betti 1 (left frames) and local Betti 2 (right frames) of three different simplicial complexes as additional 2-simplices are added.  Magenta: $\beta_k = 0$, Blue: $\beta_k = 1$, Cyan: $\beta_k = 2$.}
\label{fig:stratman}
\end{center}
\end{figure}

\begin{example}
Consider the sequence of stratified manifolds shown in Figure \ref{fig:stratman}.  Notice that the local 1-homology is nontrivial in the parts of the complex at left and center in Figure \ref{fig:stratman} that appear ``graph-like'', namely the two loops.  However, the edges in the boundary of the filled 2-simplex are identified not as having nontrivial 1-homology, which indicates that they are part of a higher-dimensional structure.  At the other extreme in the complex at right in Figure \ref{fig:stratman}, shows that the local 2-homology along the common edge among the three 2-simplices has dimension 2.  This indicates that a ramification is present there.
\end{example}

\begin{figure}
\begin{center}
\includegraphics[width=4in]{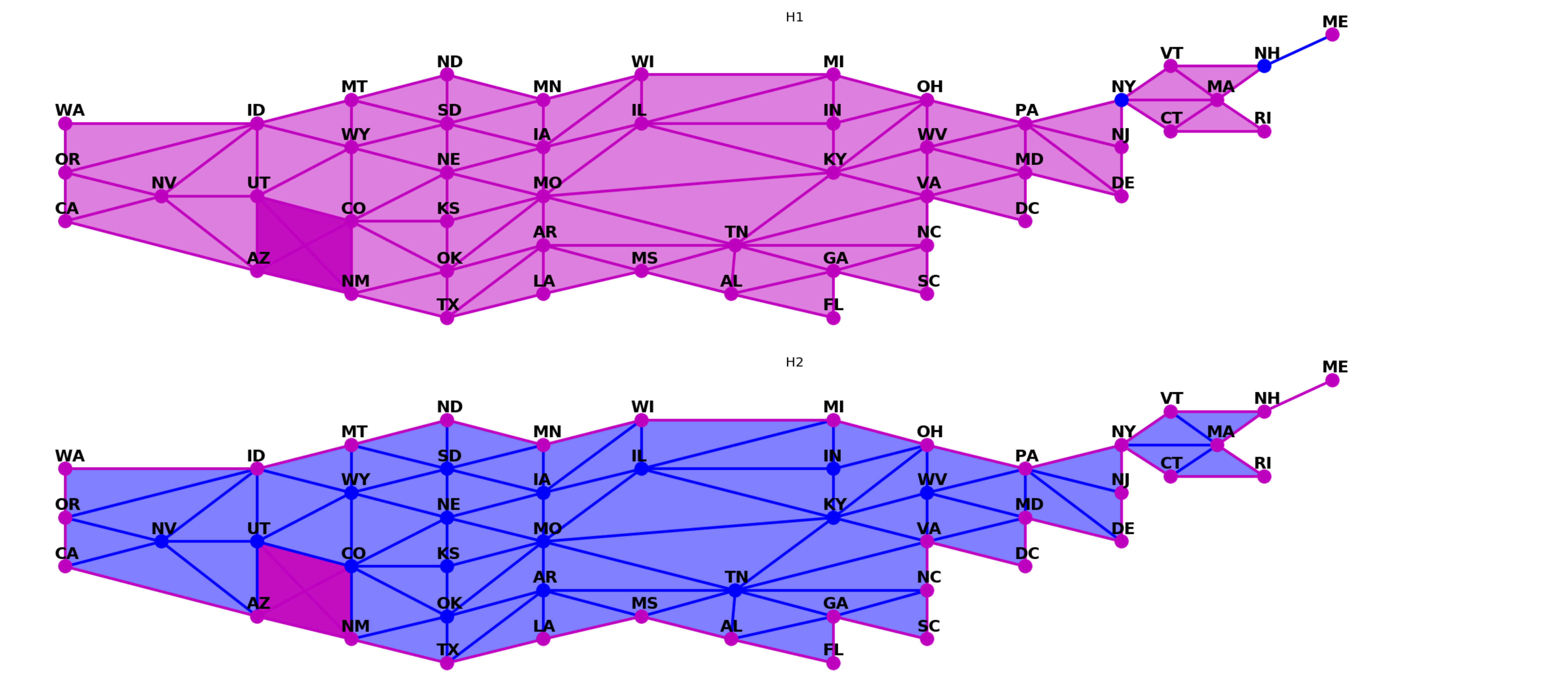}
\caption{Local Betti 1 (top) and local Betti 2 (bottom) of a map of the conterminous United States}
\label{fig:usa_map}
\end{center}
\end{figure}

\begin{example} (from \cite{Joslyn_2016})
Once the behavior of local homology on small examples is understood, it can be deployed as an analytic technique on larger complexes.  For instance, consider the abstract simplicial complex shown in Figure \ref{fig:usa_map}.  In this abstract simplicial complex, each US state corresponds to a vertex, each pair of states with a common boundary is connected with an edge, each triple of states sharing a boundary corresponds to a 2-simplex, etc.  Notice that although Maryland, Virginia, and the District of Columbia have two distinct three-way connections (the American Legion Memorial bridge and the Woodrow Wilson bridge), only one 2-simplex is present in the complex.  The presence of a common border point of Utah, Colorado, Arizona, and New Mexico is immediately and visually apparent as a change in stratification.  Additionally, the connectivity of New York and New Hampshire is easily identified as anomalous, and Maine sits on the single edge that is not included in any other simplex.
\end{example}

\subsection{Neighborhood filtration}
\label{sec:neighborhood}
For each set of faces $Y \subseteq X$ in an abstract simplicial complex, define the $0$-neighborhood\footnote{Be aware of the difference between $N(v)$ in Section \ref{sec:graphs} and $N_m([v])$ for a vertex $v$.} of $Y$ as $N_0(Y) = \st Y$ and for each $m > 0$, the $m$-neighborhood as $N_m(Y) = \st \cl N_{m-1}(Y)$.  There is a pair map $(X,N_{m-1}(Y)) \to (X,N_m(Y))$ between consecutive neighborhoods and therefore a sequence of induced maps on local homology
\begin{equation*}
  \xymatrix{
    H_k(X,N_0(Y)) \ar[r]& \dotsb H_k(X,N_{m-1}(Y)) \ar[r] & H_k(X,N_m(Y)) \ar[r]& \dotsb \\
    }
\end{equation*}
which can be thought of as a persistence module.  Observe that if $Y\subseteq Z$, then $N_m(Y) \subseteq N_m(Z)$ for all $m \ge 0$.  This means that the induced maps fit together into a commutative ladder
\begin{equation*}
  \xymatrix{
    H_k(X,N_0(Y)) \ar[r]\ar[d]& \dotsb H_k(X,N_{m-1}(Y)) \ar[r]\ar[d] & H_k(X,N_m(Y)) \ar[r]\ar[d]& \dotsb \\
    H_k(X,N_0(Z)) \ar[r]& \dotsb H_k(X,N_{m-1}(Z)) \ar[r] & H_k(X,N_m(Z)) \ar[r]& \dotsb \\
    }
\end{equation*}
This commutative ladder defines a homomorphism between persistence modules.  This means that associated to the complex $X$ is a sheaf of modules, the \emph{persistent local homology sheaf}\footnote{Beware! Persistent homology is itself a \emph{cosheaf} \cite{Curry_2015}, so we are using the adjective \emph{local} to avoid confusion.} whose stalks are persistence modules over the stars of each simplex and whose restriction maps are given by commutative ladders as above.  (Compare this construction with \cite{Bendich_2007}, which arrives at the same sheaf for Vietoris-Rips or \v{C}ech complexes.  Since they start with point cloud data, they have an additional parameter that controls the discretization.)

The situation of neighborhood filtrations is rather special, and is not functorial.  In particular, functoriality means that given a simplicial map $f : X \to X'$, we would have a sheaf morphism\footnote{Caution! Sheaf morphisms along a simplicial map ``go'' in the opposite direction from the simplicial map!  They are called \emph{$f$-cohomomorphisms} by \cite{Bredon} for this reason.} from the persistent local homology sheaf over $X'$ to the persistent local homology sheaf over $X$.  Each component map of such a morphism would have to be induced by a pair map like 
\begin{equation*}
(X',X'\setminus \st f(\sigma)) \to (X,X \setminus \st \sigma)
\end{equation*}
for each $\sigma \in X$.  But this kind of map will not be well-defined if $f$ is not bijective.  Dually, a pair map like
\begin{equation*}
(X,X\setminus \st \sigma) \to (X',X' \setminus \st f(\sigma))
\end{equation*}
also won't generally exist, but a coarsened local version does.

First, notice that a simplicial map $f:X \to X'$ descends to pair map $(X,X \backslash f^{-1}(f(\sigma))) \to (X',X'\backslash f(\sigma))$ for each $\sigma\in X$.  

Letting $Y_\sigma = f^{-1}(f(\sigma))$, we obtain
\begin{equation*}
f^{-1}( \st f(\sigma)  ) \supseteq \st ( f^{-1}(f(\sigma)) ) = \st Y_\sigma,
\end{equation*}
because $f$ is continuous.

\begin{example}
\label{eg:functorfail}
We note that the reverse inclusion does not hold in general topological spaces. If $X=X'=\{a,b\}$ but $X$ is given the discrete topology and $X'$ is given the trivial topology, then the identity map $i:X\to X'$ is continuous.  But $\st Y_a = Y_a = \{a\}$, so there does not exist a pair map $(X,X\setminus \st Y_a) \to (X',X' \setminus \st i(a)) = (X',\emptyset)$.
\end{example}

\begin{proposition}
  \label{prop:functorish}
If $f:X \to X'$ is a simplicial map, then $f$ descends to a pair map
\begin{equation*}
(X,X\setminus \st Y_\sigma) \to (X',X'\setminus \st f(\sigma)),
\end{equation*}
where $Y_\sigma = f^{-1}(f(\sigma))$ for each $\sigma\in X$.
\end{proposition}
\begin{proof}
  We need to show that $f^{-1}( \st f(\sigma)  ) \subseteq \st ( f^{-1}(f(\sigma)) )$, so suppose that $\tau \in f^{-1}( \st f(\sigma) )$. This means that $f(\tau) \in \st f(\sigma)$, which is equivalent to the statement that $f(\sigma)$ is a face of $f(\tau)$.  Suppose that $\tau = [v_0, \dotsc, v_n]$.  Since $f$ is simplicial, this means that
  \begin{equation*}
    f(\tau) = [f(v_0), \dotsc, f(v_n)]
  \end{equation*}
  (removing duplicate vertices as appropriate) and that the set of vertices for $f(\sigma)$ is a subset of $\{f(v_0), \dotsc, f(v_n)\}$.  Without loss of generality, suppose that $f(v_0)$ is a vertex in both $f(\sigma)$ and $f(\tau)$.  Thus $v_0 \in f^{-1}(f(\sigma))$ as a function on vertices, and yet $v_0$ is also a vertex of $\tau$ by assumption.  Thus $\tau \in \st [v_0] \subseteq \st f^{-1}(f(\sigma))$ as desired.
\end{proof}

Although not every open set is formed by unions of neighborhoods of $Y_\sigma$, this means that a simplicial map $f: X \to X'$ induces a map on local homology spaces
\begin{equation*}
  H_k(X,X\setminus N_m(Y_\sigma)) \to H_k(X',X'\setminus N_m(f(\sigma)))
\end{equation*}
for each $m \ge 0$.

\section{Computational considerations}\label{sec:computation}

 Our implementation of the computation of local homology is focused on computing relative homology of an abstract simplicial complex at an arbitrary simplex.  Our implementation was written using Python 2.7 and uses the {\tt numpy} library and is available as an open-source module of the {\tt pysheaf} repository on GitHub \cite{pysheaf}.  The use of {\tt numpy} simplifies the linear algebraic calculation, but does require that all calculations are performed using double-precision floating point arithmetic rather than $\mathbb{Z}$.  This means that torsion cannot be computed, but none of the theoretical results presented in this article rely upon torsion.  Steps in the description below that depend upon {\tt numpy} are noted.

The abstract simplicial complex $X$ is stored as a list of lists of vertices.  Each list of vertices represents a simplex, though all of its faces are included implicitly.   
The ordering of the list of vertices induces a total order on the vertices which in turn induces total orderings on the simplices. 
In particular, for runtime efficiency, the ordering of vertices within a simplex is required to be consistent with a fixed total ordering of vertices.  (This assumption is not enforced in our implementation, though incorrect results will be obtained if it is violated.)

 We made extensive use of Python \emph{dictionaries} since Python accesses dictionaries in constant time.  For comparison, we also wrote a version in which lists were used in place of dictionaries.  By expunging unnecessary list accesses, we were obtained substantial runtime reductions as shown in Table \ref{tab:profile_dictionaries}.  The results in the Table were obtained using an Intel Core i7-4900MQ running at 2.80 GHz with 32 GB DDR3 RAM on Windows 7.  Although using dictionaries does result in a performance penalty during the construction of neighborhoods of simplices, the overall runtime improvements are substantial.

\begin{table}
  \begin{center}
    \caption{Runtime reduction due to dictionaries in Stage 1}
    \begin{tabular}{|l|c|c|c|}
      \hline
      Example & List & Dict & Percent \\
      & runtime (s) & runtime (s) & decrease \\
      \hline
      USA (Figure \ref{fig:usa_map})&45.860&0.871&98.10\\
      Annulus (Figure \ref{fig:annulus})&223.488&4.106&98.16\\
      Random complex (Figure \ref{fig:thick_graph})&345.790&3.285&99.05\\
      Karate graph (Figure \ref{fig:graphs})&210.825&9.508&95.49\\
      \hline
    \end{tabular}
    \label{tab:profile_dictionaries}
    \end{center}
\end{table}

 For storage efficiency, it is only necessary to store maximal simplices, those that are not included in any higher-dimensional simplex.  We make the assumption that $X$ explicitly lists only simplices that are not included in any others.  This assumption has a runtime penalty, since faces of simplices will need to be computed as needed.  On the other hand, only faces of a certain dimension and of certain simplices will need to be computed at any given time.

\section{Statistical comparison with graph invariants}
\label{sec:statistics}

In this section, we compare several popular local invariants of graphs with the local homology of the flag complex on such graphs.  We first provide the definitions of the graph invariants that we consider in Section \ref{sec:ns_definitions}, noting that they are either \emph{vertex-based} or \emph{edge-based}.  Our comparison methodology is outlined in Section \ref{sec:ns_method}.  Briefly, comparison between a vertex-based (or edge-based) graph invariant and local homology at a vertex (or edge) is straightforward.  For other simplices in the flag complex the graph invariant must be extended in some fashion as we describe in that Section.  Section \ref{sec:ns_datasets} introduces the datasets we used for comparison.  Section \ref{sec:ns_comparison_measure} discusses our results.

\subsection{Graph invariants used in our comparison}
\label{sec:ns_definitions}
The following seven invariants are defined for an undirected graph $G = (V,E)$, where $V$ is the vertex set, and $E$ is the set of undirected edges:
\begin{enumerate}
\item \emph{Degree centrality} \cite{Diestel_2012},
\item \emph{Closeness centrality} \cite{Bavelas_1950},
\item \emph{Vertex and Edge betweenness centrality} \cite{Freeman_1977, Brandes_2001},
\item \emph{Random walk vertex betweenness centrality} \cite{Newman_2005},
\item \emph{Maximal clique count} \cite{Diestel_2012}, and
\item \emph{Clustering coefficient} \cite{Watts_1998}.
\end{enumerate}
Apart from the edge betweenness centrality which is defined for each edge in the graph, all the other invariants are defined on the vertices of the graph.

\subsection{Comparison methodology}
\label{sec:ns_method}

In contrast to graph-based invariants like betweenness centrality which measure how central a node is in the context of the whole graph, local homology ignores all but the local neighborhood and enumerates topological features of that neighborhood.  We restrict the comparisons to vertices and edges for which the graph invariants are well defined. For edges, we also consider the aggregation of the vertex-based graph invariants corresponding to the two vertices constituting the edge via averaging. Formal extension of the comparison methodology to higher-order faces beyond edges is being considered through appropriate contraction of the corresponding faces to a \textit{super vertex}.   

Consider a $(k-1)$-dimensional face $\sigma \in F(G)$ of the flag complex (Definition \ref{def:flag_complex}) constructed from the corresponding set of graph vertices $\{v_1,v_2,\dots,v_k\}$. Let $f(v)$ denote a particular vertex-based graph invariant for a given vertex $v$. 

We consider two ways of comparing the graph invariants with local homology are as follows:

\begin{enumerate}

\item If $\sigma=[v]$, then we may compare $f(v)$ with $\beta_k(\st v)$ directly.

\item If $\sigma=[v_1,v_2]$ and $[v_1,v_2] \in E$, where $E$ is the set of edges, we compare $\frac{\left(f(v_1)+f(v_2)\right)}{2}$ with $\beta_k(\st [v_1,v_2])$ 

\end{enumerate}

We present scatter plots for comparing the local homologies with various graph invariants for three different graphs (described in Section \ref{sec:ns_datasets}). Three different neighborhoods $N_0$, $N_1$ and $N_2$ have been considered for the computation of the local homologies. For the edges, we considered edge-betweenness centrality and also the aggregation of the graph invariants corresponding to the two vertices constituting the edge via averaging.

The two invariants under consideration are correlated using Pearson's correlation coefficient ($\rho$).  Given two real-valued vectors of same length, $X$ and $Y$, Pearson's correlation coefficient $\rho_{X,Y}$ is given as follows.
\begin{equation*}
\rho_{X,Y} = \frac{cov(X,Y)}{\sigma_X\sigma_Y}
\end{equation*}
Here \textit{cov} refers to the co-variance and $\sigma$ refers to the standard deviation respectively.

For each of the graphs, we present results corresponding to a subset of the graph invariants and local Betti numbers for which good to excellent ($|\rho|$ = 0.6 to 0.9) correlation is observed.  Surprisingly we observed \textit{very little} correlation ($\rho$ was highly variable in magnitude and sign and was between 0 and 0.4) between $\beta_1$ for the edges for the $N_0$ and $N_1$ neighborhoods with the edge betweenness centrality which is directly calculated for each edge on the graph without any aggregation steps. Hence, for the edges, we only present the results corresponding to the aggregation of the vertex-based invariants via averaging. 

\subsection{Dataset description}
\label{sec:ns_datasets}
In this study we focus on three different graphs namely
\begin{enumerate}
\item The well-known Zachary Karate Club Network graph with 34 vertices and 78 edges \cite{Zachary_1977},
\item A synthetic Erd\H{o}s-R\'{e}nyi graph with 40 vertices and 146 edges \cite{Erdos_1959}, and 
\item A synthetic Barabasi-Albert preferential attachment graph with 40 vertices and 144 edges \cite{Barabasi_1999}.
\end{enumerate}

The synthetic graphs were drawn from two different families with highly dissimilar degree distributions. All of the graphs are connected. A visualization of the three graphs is shown in Figure \ref{fig:graphs}. The visualizations were created by using the {\tt Gephi} software package \cite{gephi}. 

\begin{figure}[!htb]
\includegraphics[width=\textwidth]{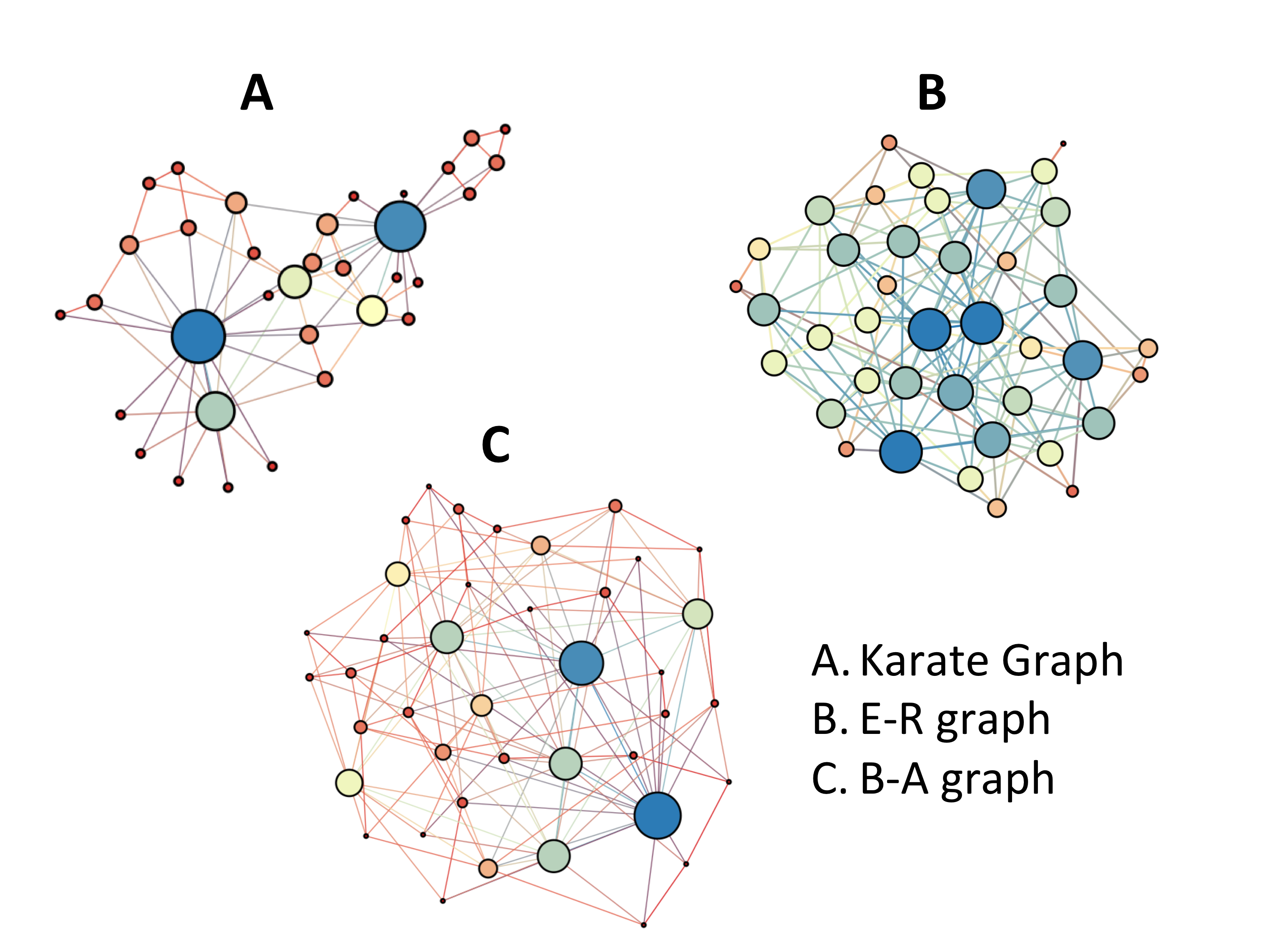}
\centering
\caption{The three graphs used in the comparison study in this article.  The sizes and colors of the vertices correspond to their degree.}
\label{fig:graphs}
\end{figure}

\subsection{The Karate graph}

For the Karate graph we observed very good positive correlation between the local Betti number $\beta_1$ for the $N_0$  neighborhood with a number of vertex specific graph invariants. Additionally good negative correlation was observed with the local clustering coefficient.

\begin{table}
  \begin{center}
    \caption{Karate graph correlations between vertex-centered local homology and other invariants}
    \label{tab:karate_vert}
\begin{tabular}{r|ccc|ccc}
Karate graph      & $\beta_1(N_0)$ & $\beta_1(N_1)$ & $\beta_1(N_2)$ & $\beta_2(N_0)$ & $\beta_2(N_1)$ & $\beta_2(N_2)$ \\\hline
Degree centrality &		{\bf 0.700}		&	{\bf 0.700}		  &		0.422		 &		0.520	   &		0.520	 &	-0.001		  \\
Closeness centrality &	{\bf 0.726}			&	{\bf 0.726}		  &	{\bf 0.703}			 &	0.349		   &	0.349		 &	0.339		  \\
Betweenness centrality ($v$) &		{\bf 0.741}	    &	{\bf 0.741}		  &		0.311		 &		{\bf 0.740}	   &	{\bf 0.740}		 &		-0.031	  \\
Random walk centrality &	{\bf 0.761}			&		{\bf 0.761}	  &		0.388		 &	{\bf 0.644}		   &	{\bf 0.644}		 &	0.011		  \\
Maximal cliques     &		{\bf 0.718}		&	{\bf 0.718}		  &		0.307		 &		0.548	   &		0.548	 &	-0.085		  \\
Clustering coeff.     &		{\bf -0.656}		&	{\bf -0.656}		  &		-0.154		 &	-0.214		   &		-0.214	 &		0.005	  \\
\end{tabular}
  \end{center}
\end{table}

\begin{table}
  \begin{center}
    \caption{Karate graph correlations between edge-centered local homology and aggregation of other invariants}
    \label{tab:karate_edge}
\begin{tabular}{r|ccc|ccc}
Karate graph      & $\beta_1(N_0)$ & $\beta_1(N_1)$ & $\beta_1(N_2)$ & $\beta_2(N_0)$ & $\beta_2(N_1)$ & $\beta_2(N_2)$ \\\hline
Degree centrality &		-0.026		&	{\bf 0.740}		  &		0.092		 &		0.434	   &	0.391		 &		-0.167	  \\
Closeness centrality &		0.116		&		{\bf 0.677}	  &		0.536		 &	0.226		   &	0.345		 &	0.342		  \\
Betweenness centrality ($v$) &		0.013	    &	{\bf 0.706}		  &		0.005		 &		0.339	   &	{\bf 0.641}		 &	-0.317		  \\
Random walk centrality &	0.013			&	{\bf 0.759}		  &		0.080		 &		0.392	   &	0.578		 &	-0.206		  \\
Maximal cliques &		0.010		&		{\bf 0.743}	  &		0.020	 &	0.434   &	0.391		 &	-0.276		  \\
Clustering coeff. &		-0.418		&		{\bf -0.639}	  &		-0.348		 &		-0.259	   &	-0.141		 &		-0.055	  \\

Betweenness centrality ($e$) &		0.283		&	0.358  &		0.151	 &		-0.018	   &	0.406   &	-0.123		  \\
\end{tabular}
  \end{center}
  \end{table}

The results are summarized in Tables \ref{tab:karate_vert} and \ref{tab:karate_edge}, and Figure \ref{fig:VLH1N0Karate} for the $N_0$ case. 

\begin{figure}[!htb]
\includegraphics[width=\textwidth]{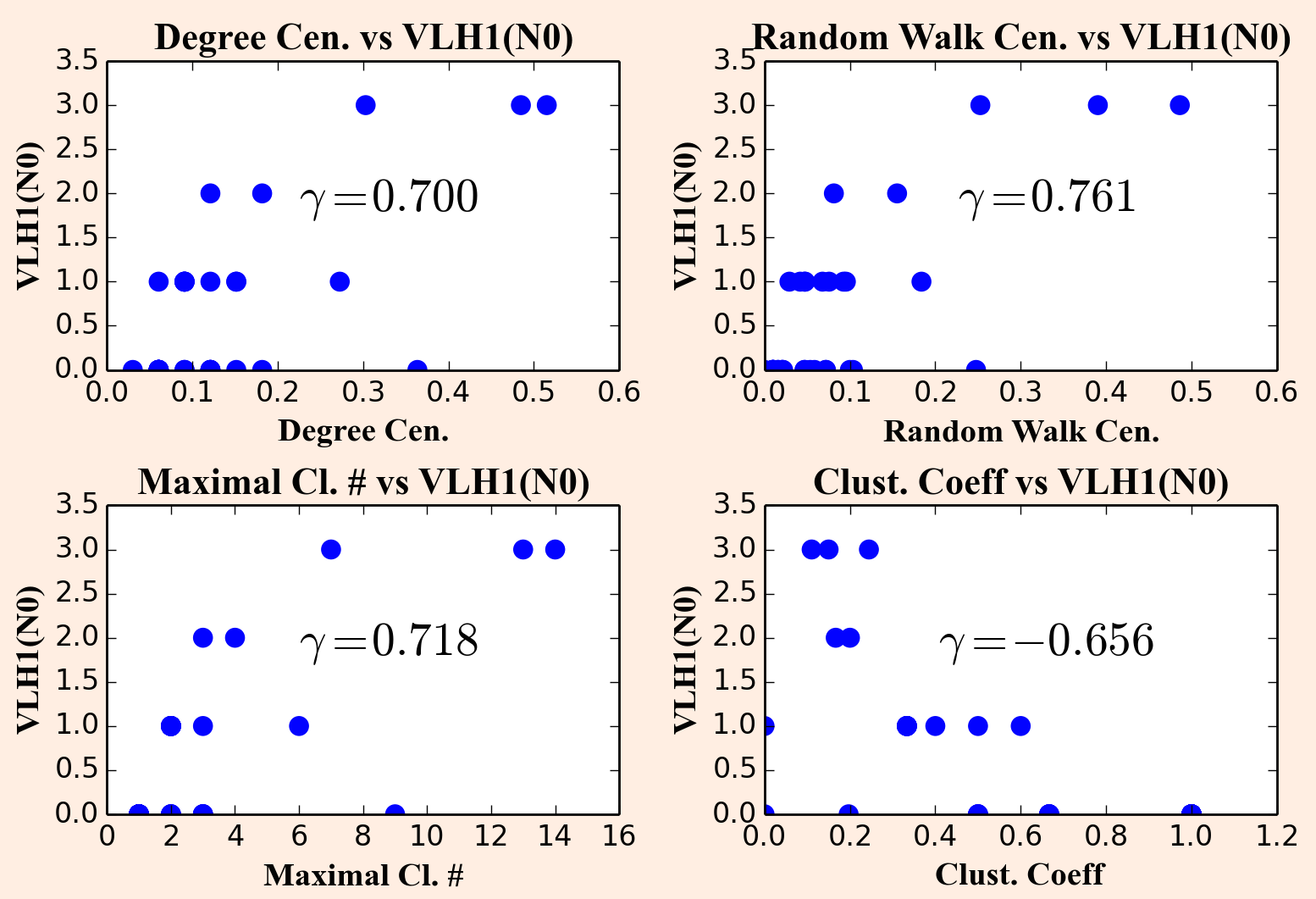}
\centering
\caption{Scatter plots comparing the various vertex-specific graph invariants with $\beta_1$ for the $N_0$ neighborhood for the Karatex graph.}
\label{fig:VLH1N0Karate}
\end{figure}

\subsection{The Erd\H{o}s-R\'{e}nyi graph}
For the Erd\H{o}s-R\'{e}nyi graph with 40 nodes and 146 edges, we observed excellent correlation between the various centrality values (including the maximal clique count) and $\beta_1$ for the $N_1$ neighborhood as shown in Tables \ref{tab:er_vert} and \ref{tab:er_edge}.

\begin{table}
  \begin{center}
    \caption{Erd\H{o}s-R\'{e}nyi graph correlations between vertex-centered local homology and other invariants}
    \label{tab:er_vert}
\begin{tabular}{r|ccc|ccc}
ER(40) graph      & $\beta_1(N_0)$ & $\beta_1(N_1)$ & $\beta_1(N_2)$ & $\beta_2(N_0)$ & $\beta_2(N_1)$ & $\beta_2(N_2)$ \\\hline
Degree centrality &		0.163		&	0.163		  &		{\bf 0.901}		 &		0.474	   &	0.474		 &		0.030	  \\
Closeness centrality &	0.181			&	0.181		  &		{\bf 0.907}		 &	0.429		   &	0.429		 &		0.059	  \\
Betweenness centrality ($v$) &		0.229	    &	0.229		  &	{\bf 0.727}			 &		0.316	   &	0.316		 &	-0.143		  \\
Random walk centrality &	0.303			&	0.303		  &		0.846		 &	0.353		   &	0.353		 &	-0.025		  \\
Maximal cliques     &		0.101		&	0.101		  &		{\bf 0.857}		 &	0.624		   &	0.624		 &		0.088	  \\
Clustering coeff.     &		{\bf -0.718}		&	{\bf -0.718}		  &		-0.218		 &		0.126	   &	0.426		 &	-0.118		  \\
\end{tabular}
  \end{center}
  \end{table}

\begin{table}
  \begin{center}
    \caption{Erd\H{o}s-R\'{e}nyi graph correlations between edge-centered local homology and aggregation of other invariants}
    \label{tab:er_edge}
\begin{tabular}{r|ccc|ccc}
ER(40) graph      & $\beta_1(N_0)$ & $\beta_1(N_1)$ & $\beta_1(N_2)$ & $\beta_2(N_0)$ & $\beta_2(N_1)$ & $\beta_2(N_2)$ \\\hline
Degree centrality &		-0.322		&		-0.115	  &		{\bf 0.836}		 &		0.394	   &	0.299		 &		-0.227	  \\
Closeness centrality &	-0.319			&	-0.095		  &		{\bf 0.842}	 &		0.380	   &	0.263		 &	-0.243		  \\
Betweenness centrality ($v$) &	-0.206		    &	-0.006	  &		{\bf 0.631}		 &	0.276		   &	0.113		 &	-0.295		  \\
Random walk centrality &		-0.223		&	0.049		  &		{\bf 0.726}		 &	0.320		   &	0.164		 &		-0.227	  \\
Maximal cliques &	-0.256			&	-0.210		  &		{\bf 0.787}		 &	0.491		   &	0.433		 &		-0.121	  \\
Clustering coeff. &	-0.528			&	{\bf -0.794}		  &		-0.120		 &	0.233		   &	0.196		 &	-0.239		  \\
Betweenness centrality ($e$) &		0.398		&	-0.035		  &		-0.130		 &	-0.368   &	-0.001  &		-0.115	  \\
\end{tabular}
\end{center}
\end{table}

Figure \ref{fig:VLH1N1ER40} shows the high correlation between $\beta_1$ and centrality for the $N_1$ neighborhood. However the correlation with the local clustering coefficient was very bad for the same scenario. We also noted that the correlation of $\beta_1$ with the local clustering clustering coefficient was very good for the $N_0$ neighborhood (high negative value). 

\begin{figure}[!htb]
\includegraphics[width=12cm]{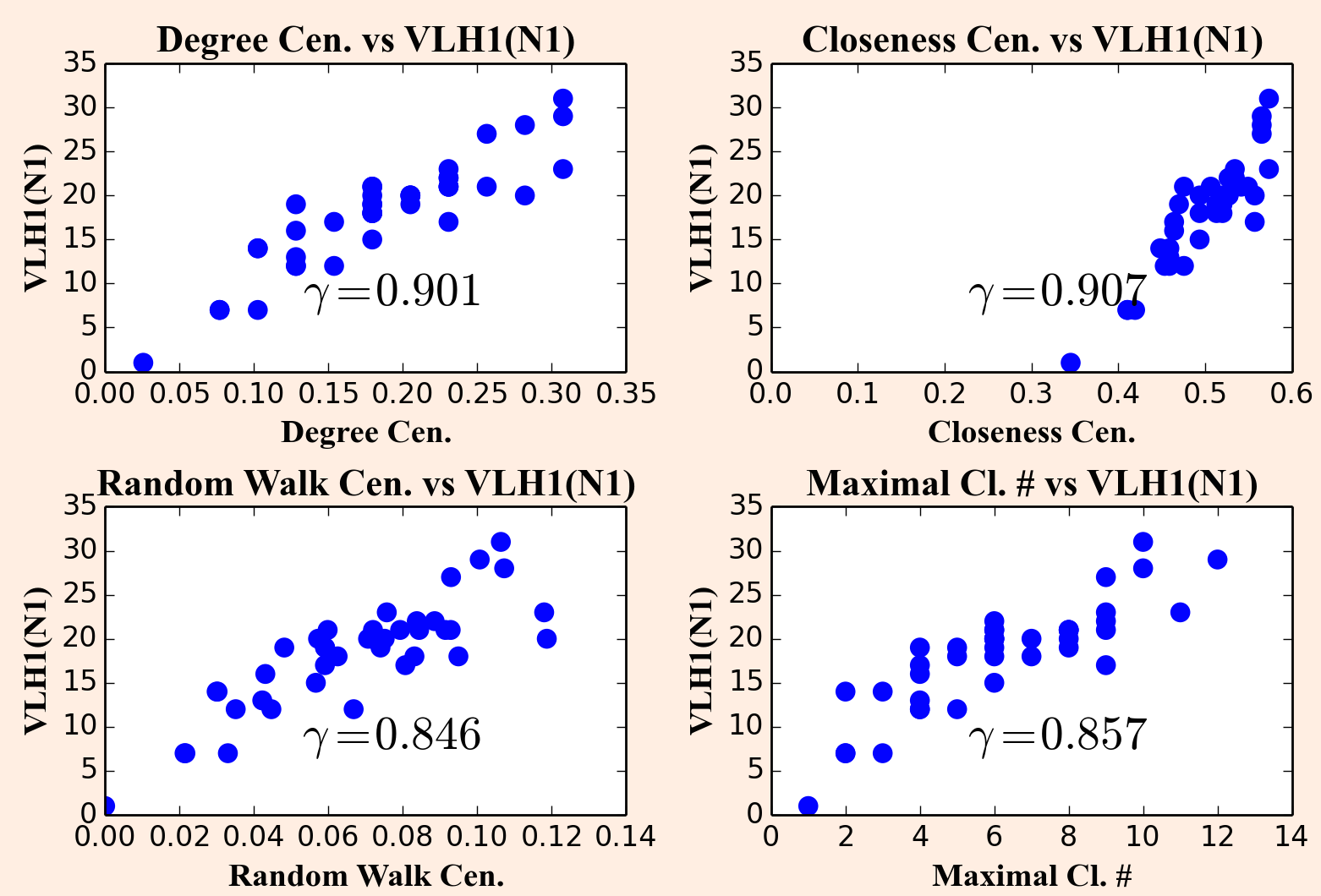}
\centering
\caption{Scatter plots comparing the vertex graph invariants for the Erd\H{o}s-R\'{e}nyi graph with $\beta_1$ for the $N_1$ neighborhood.}
\label{fig:VLH1N1ER40}
\end{figure}

\subsection{The Barabasi-Albert graph}
Finally we considered the Barabasi-Albert preferential attachment with 40 nodes and 144 edges and ran similar comparisons, as shown in Tables \ref{tab:ba_vert} and \ref{tab:ba_edge}.

\begin{table}
  \begin{center}
    \caption{Barabassi-Albert graph correlations between vertex-centered local homology and other invariants}
    \label{tab:ba_vert}
\begin{tabular}{r|ccc|ccc}
BA(40) graph      & $\beta_1(N_0)$ & $\beta_1(N_1)$ & $\beta_1(N_2)$ & $\beta_2(N_0)$ & $\beta_2(N_1)$ & $\beta_2(N_2)$ \\\hline
Degree centrality &		-0.114		&		-0.114	  &		{\bf 0.839}		 &		{\bf 0.844}	   &		{\bf 0.844}	 &		-0.161	  \\
Closeness centrality &	-0.171			&	-0.171		  &		{\bf 0.788}		 &	{\bf 0.797}		   &	{\bf 0.797}		 &		-0.007	  \\
Betweenness centrality ($v$) &	-0.034		    &	-0.034		  &		{\bf 0.798}		 &	{\bf 0.849}		   &	{\bf 0.849}		 &	-0.144	 \\
Random walk centrality &	-0.033			&		-0.033	  &		{\bf 0.828}		 &	{\bf 0.800}		   &	{\bf 0.800}		 &		-0.152	  \\
Maximal cliques     &		-0.137		&	-0.137		  &		0.827		 &	{\bf 0.915}		   &		{\bf 0.915}	 &		-0.229	  \\
Clustering coeff.     &		{\bf -0.657}		&	{\bf -0.657}		  &		-0.533		 &	-0.224		   &	-0.224		 &	0.244		  \\
\end{tabular}
  \end{center}
  \end{table}

\begin{table}
  \begin{center}
    \caption{Barabassi-Albert graph correlations between edge-centered local homology and aggregation of other invariants}
    \label{tab:ba_edge}
\begin{tabular}{r|ccc|ccc}
BA(40) graph      & $\beta_1(N_0)$ & $\beta_1(N_1)$ & $\beta_1(N_2)$ & $\beta_2(N_0)$ & $\beta_2(N_1)$ & $\beta_2(N_2)$ \\\hline
Degree centrality &		-0.302		&	-0.211		  &		{\bf 0.815}		 &		0.564	   &	{\bf 0.830}		 &	-0.594		  \\
Closeness centrality &	-0.295			&	-0.242	  &		{\bf 0.818}		 &		0.513	   &	{\bf 0.816}		 &	-0.563		  \\
Betweenness centrality ($v$) &	-0.226 &		-0.085	  &	{\bf 0.738}			 &		0.487	   &	{\bf 0.825}		 &	-0.519		  \\
Random walk centrality &		-0.273		&	-0.126		  &		{\bf 0.780}		 &		0.545	   &	{\bf 0.788}		 &	-0.555		  \\
Maximal cliques &		-0.276		&	-0.243		  &		{\bf 0.809}		 &	0.590		   &	{\bf 0.882}		 &		{\bf -0.647}	  \\
Clustering coeff. &		-0.256		&	-0.561		  &		-0.341		 &	-0.156		   &	-0.058		 &	0.172		  \\
Betweenness centrality ($e$)  &	 0.053	 &	 0.163   &	 0.421	  &	  0.012		 &	  0.462   &	   -0.274    \\
\end{tabular}
\end{center}
\end{table}

For the vertex specific invariants (Table \ref{tab:ba_vert}), we found excellent correlation between the centrality invariants (including the maximal clique count) for $\beta_1$ for the $N_1$ neighborhood but as with the Erd\H{o}s-R\'{e}nyi graph, the correlation with clustering coefficient was bad. Clustering coefficient on the other hand was again well correlated with $\beta_1$ for the $N_0$ neighborhood. Figure \ref{fig:VLH1N1BA40} captures the correlations for $\beta_1$ for the $N_1$ neighborhood.

\begin{figure}[!htb]
\includegraphics[width=\textwidth]{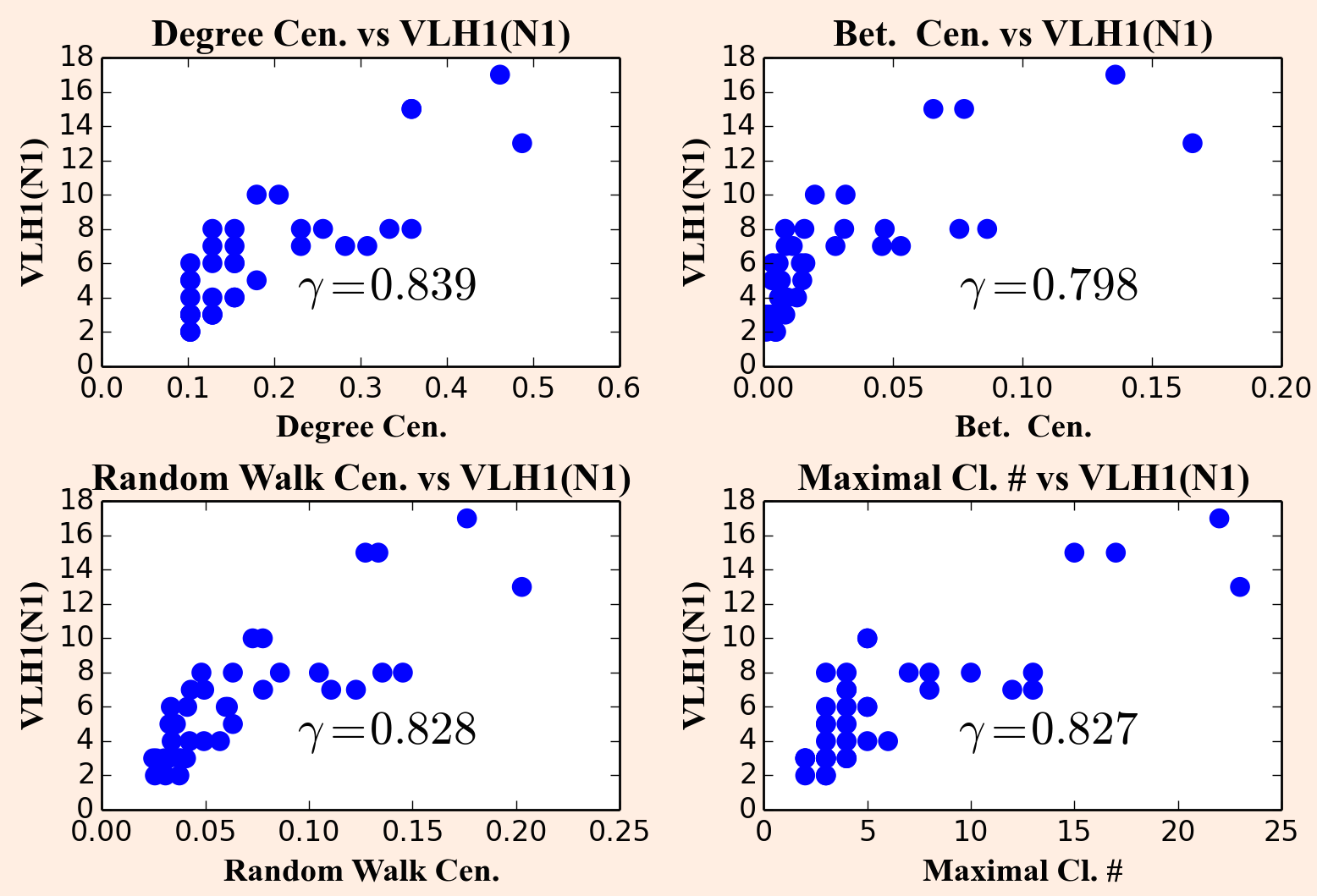}
\centering
\caption{Scatter plots comparing the vertex graph invariants for the Barabasi-Albert graph with $\beta_2$ for the $N_1$ neighborhood.}
\label{fig:VLH1N1BA40}
\end{figure}

\subsection{Summary and discussions}
\label{sec:ns_comparison_measure}
We ran a large number of combinations in the correlation study and noticed a few specific trends and the same are summarized below.  Generally, there are some strong correlations between the local Betti number and various graph invariants, but the local Betti is also clearly quite distinct.  We therefore conclude that it provides \textit{independent information} about the local structure of a graph.

\paragraph{Local Clustering Coefficient:}

Figure \ref{fig:CCGoodN1} specifically shows three scatter plots, one each for each of the three graphs considered, comparing the local clustering coefficient with the vertex local homology $\beta_1$ for the $N_0$ neighborhood. Thus it can be seen that $\beta_1(N_0)$ is typically well correlated in a negative sense with the local clustering coefficient. This is in line with the observation that higher the clustering coefficient for a given vertex, the higher the chance of existence of neighborhood triangles which then will reduce the possibility of open loops thereby leading to lower $\beta_1$ values.

\begin{figure}[!htb]
\includegraphics[width=\textwidth]{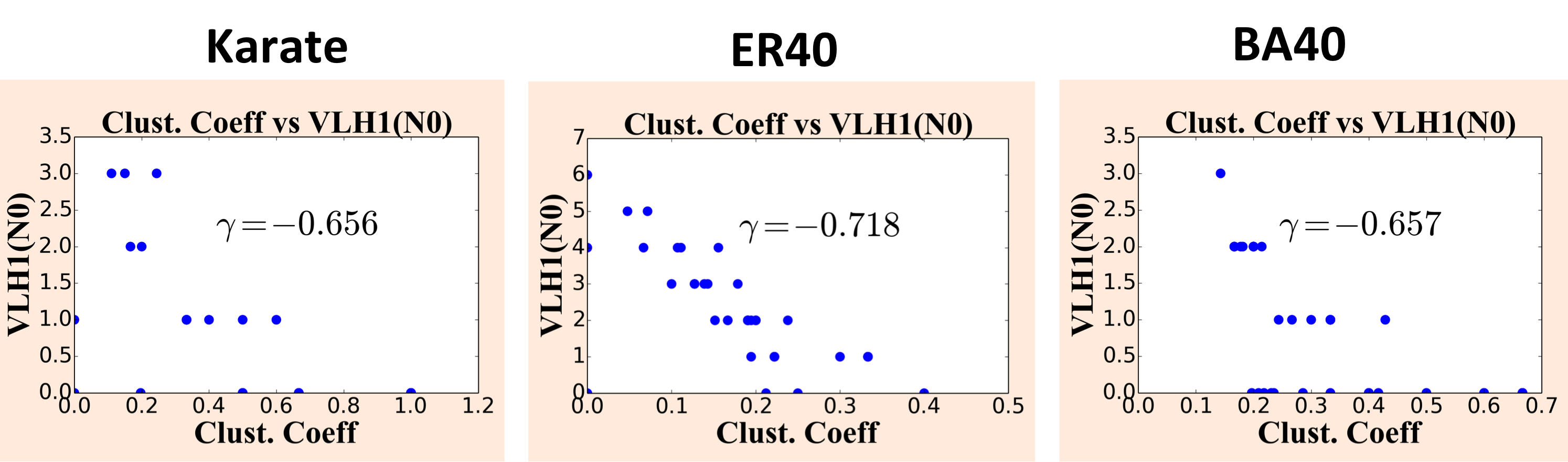}
\centering
\caption{Scatter plots comparing the local clustering coefficient and $\beta_1$ for $N_1$ neighborhood across all three graphs considered. A, B, and C denote the Karate graph, ER(40) denotes the Erd\H{o}s-R\'{e}nyi graph, and BA(40) denotes the Barabasi-Albert graph respectively. Good correlations (-0.656,-0.718,-0.657) are obtained for A, B, and C respectively.}
\label{fig:CCGoodN1}
\end{figure}

\paragraph{Centrality based invariants:}

While the various centrality based invariants (including the maximal clique count) showed moderate to good correlation with the local Betti number $\beta_1$ for $N_0$ and $N_1$ neighborhoods, we specifically noticed very good correlation ($|\rho|$ up to 0.9) with $\beta_1$ for the $N_1$ neighborhood for both of the synthetic graphs, and the $N_1$ neighborhood for the Karate graph. Many results along this line were presented in the earlier subsections. Specifically we notice that $\beta_1$ tends to be correlated positively with degree centrality since a higher degree can result in a higher possibility of forming open loops and it is known from network science literature that most of the vertex centrality measures are positively correlated with the degree centrality. 

\section{Future directions}

At present, there are very few software libraries available that are capable of computing local or relative homology.  Aside from our own {\tt pysheaf} \cite{pysheaf}, we are only aware that {\tt RedHom} \cite{RedHom} is able to compute relative homology.  There is considerable need for the equivalent of reductions or coreductions for relative homology to improve computational efficiency.  This is likely to be fraught with difficulties as reductions that are useful in one neighborhood may not be useful in another.

How robust to noise is the local homology of a combinatorial space? If simplices are included with some probability distribution how does that affect the local homology?  At present, results are avialable for the global homology of random simplicial complexes as the number of simplices grows \cite{Kahle_2009}, but this says nothing of its local homology.  Additionally, while persistent local homology of point clouds is now an active area of study, it is yet unclear how applicable the robustness theorems obtained (for instance \cite{Bendich_2007}) relate to general filtrations of combinatorial spaces.

\section*{Acknowledgements}
Partial funding for this work was provided by DARPA SIMPLEX N66001-15-C-4040.

\bibliographystyle{plain}
\bibliography{localhomology_bib}
\end{document}